\documentclass[11pt]{amsart}

\usepackage[T1]{fontenc}
\usepackage[utf8]{inputenc}
\usepackage{lmodern}
\usepackage{microtype}
\usepackage{mathtools,amssymb,amsmath,amsthm}
\usepackage{booktabs,longtable,array}
\usepackage{enumitem}
\usepackage{xcolor}
\usepackage{hyperref}
\usepackage[nameinlink,capitalize,noabbrev]{cleveref}
\usepackage{geometry}
\usepackage{listings}
\usepackage{tikz}
\usetikzlibrary{calc}
\hypersetup{colorlinks=true,linkcolor=blue!55!black,citecolor=blue!55!black,urlcolor=blue!55!black}

\lstdefinestyle{codeappendix}{
  basicstyle=\ttfamily\scriptsize,
  stepnumber=1,
  numbersep=6pt,
  showstringspaces=false,
  breaklines=true,
  breakatwhitespace=false,
  columns=fullflexible,
  keepspaces=true,
  frame=single,
  xleftmargin=1.2em,
  framexleftmargin=0.8em,
  aboveskip=0.8em,
  belowskip=0.8em,
  literate={ℝ}{{$\mathbb{R}$}}1 {∧}{{$\wedge$}}1 {≤}{{$\le$}}1 {⊢}{{$\vdash$}}1 {⟨}{{$\langle$}}1 {⟩}{{$\rangle$}}1 {·}{{$\cdot$}}1 {×}{{$\times$}}1
}
\lstdefinestyle{LeanCode}{
  style=codeappendix,
  language={}
}
\lstdefinestyle{PythonCode}{
  style=codeappendix,
  language=Python
}

\newtheorem{theorem}{Theorem}[section]
\newtheorem{proposition}[theorem]{Proposition}
\newtheorem{lemma}[theorem]{Lemma}
\newtheorem{corollary}[theorem]{Corollary}

\theoremstyle{definition}
\newtheorem{definition}[theorem]{Definition}
\newtheorem{remark}[theorem]{Remark}

\crefname{hypothesis}{Hypothesis}{Hypotheses}
\Crefname{hypothesis}{Hypothesis}{Hypotheses}
\crefname{status}{Proof-status note}{Proof-status notes}
\Crefname{status}{Proof-status note}{Proof-status notes}

\DeclareMathOperator{\conv}{conv}
\DeclareMathOperator{\area}{area}
\DeclareMathOperator{\len}{len}
\newcommand{\interior}{\operatorname{int}}
\newcommand{\RR}{\mathbb R}
\newcommand{\SSS}{\mathbb S^1}
\newcommand{\Rhat}{\widehat R}
\newcommand{\eps}{\varepsilon}
\newcommand{\preceqT}{\preceq}
\newcommand{\precT}{\prec}

\title{Balanced Support Calibrations for Moser's Worm Problem: Exact Certificate and Bound of Triangular Cover}
\author{Zhipeng Deng}
\makeatletter

\def\@settitle{%
  \begin{center}%
  \baselineskip14\p@\relax
  \bfseries
  \@title
  \end{center}%
}

\def\@setauthors{%
  \begingroup
  \trivlist
  \centering
  \footnotesize
  \@topsep30\p@\relax
  \advance\@topsep by -\baselineskip
  \item\relax
  \author@andify\authors
  \def\\{\protect\linebreak}%
  \authors
  \endtrivlist
  \endgroup
}

\makeatother

\begin{document}
\begin{abstract}
Moser's worm problem asks for a planar region of minimum area containing a
congruent copy of every rectifiable planar arc of length one.  We study the
Bellman's lost-in-a-forest problem escape path for the
isosceles triangle
\[
 T=\text{conv}\{(-c,0),(c,0),(0,s)\},\qquad
 s=\frac{766}{\sqrt{625565}},\qquad c=\frac{197}{\sqrt{625565}}.
\]
The paper gives an exact positive four-source support calibration, a continuum
to standard-polygonal reduction, a selected $\Lambda$-gap estimate, a corrected
high-angle side-meeting argument, and exact finite ledgers.  The original
inner-anchor ledger covers 25 temporal orders: its 275 nonzero suffix states
reduce to 13 exact squared norms and satisfy $\|R\|<27131/25000$.
If either near anchor fails, a new delimiter-gap lemma produces one or two
omitted hull edges whose normals range over an exact compact fan interval.
A dependency-free rational replay checks 512 one-delimiter orders over 1540
signed half-angle intervals and 4096 two-delimiter orders over a finite signed
rectangle cover.  These three exhaustive branches give unconditionally
\[
 E(T)\ge D:=\frac{82074390584}{84861020075}
 =0.9671624323094728\ldots
\]
Hence, the convex universal cover of Moser's worm 
\[
 \frac{\text{area}(T)}{D^2}
 =\frac{11511821678274125}{44639604443512928}
 =0.257883595112076188\ldots.
\]

\end{abstract}
\maketitle

\pagestyle{plain} 

\textbf{2020 Mathematics Subject Classification.}
52A40, 52C15, 49Q10, 49K30, 90C25.

\textbf{Keywords.}
Moser's worm problem, Bellman's lost-in-a-forest problem, universal cover, support function, convex geometry
\tableofcontents

\section{Introduction}

Moser's worm problem originated in Leo Moser's 1966 collection of unsolved
questions in combinatorial geometry and asks for the planar region of smallest
area that can accommodate every rectifiable planar arc of length one.  A
unit-length rectifiable arc is traditionally called a \emph{worm}.  The
historical formulation is convex, although later work also considers
unrestricted nonconvex covers; these are distinct optimization problems.  We
write
\[
 M_{\rm conv}:=\inf\{\area(K):K\subset\RR^2\text{ compact, convex, and covering every unit worm}\},
\]
and write $M$ for the unrestricted infimum.  Every convex construction is also
an unrestricted construction, so $M\le M_{\rm conv}$.

The upper-bound literature developed through a sequence of explicit universal
covers.  Gerriets and Poole studied convex regions covering arcs of constant
length \cite{GerrietsPoole}.  Norwood, Poole, and Laidacker obtained a cover of
area below $0.27524$ \cite{NorwoodPooleLaidacker}; Wang later reduced the
then-known convex upper bound to $0.270911861$ \cite{Wang}.  In the unrestricted
setting Norwood and Poole produced a nonconvex cover of area $0.260437$
\cite{NorwoodPoole}, subsequently reduced to $0.26007$ by Ploymaklam and
Wichiramala \cite{PloymaklamWichiramala}.  Wetzel's conjectured $30^\circ$ unit
sector has area $\pi/12=0.261799\ldots$; after partial results for drapeable arcs
\cite{MakiWetzelWichiramala,MovshovichWetzel}, Panraksa and Wichiramala proved
that the sector covers every unit arc \cite{PanraksaWichiramala}.  In 2026,
Wichiramala and Panraksa gave a computer-assisted proof for Wetzel's distinct
$30^\circ$--$60^\circ$--$90^\circ$ triangle and a certified homothetic shrink of
area approximately $0.260956$ \cite{WichiramalaPanraksa}.

Lower bounds have improved more slowly.  Schaer's broadest-curve construction
underlies the classical lower bound near $0.2194$ \cite{Schaer,Wetzel}.
Khandhawit and Sriswasdi raised the convex lower bound to $0.227498$ using
forced placements of explicit worms \cite{KhandhawitSriswasdi}; Khandhawit,
Pagonakis, and Sriswasdi later proved
\[
 M_{\rm conv}\ge0.232239
\]
by a sharpened min--max construction \cite{KhandhawitPagonakisSriswasdi}.
Thus a substantial gap remains between lower and upper certificates.

A structural warning is essential.  Finite polygonal complexity does not
characterize the continuum problem: for every fixed $n$, Panraksa, Wetzel, and
Wichiramala constructed a convex region covering every $n$-segment unit
polygonal arc while failing to cover all unit arcs
\cite{PanraksaWetzelWichiramala}.  Hence a finite computation proves a worm
upper bound only when it is coupled to a theorem closing the continuum gap.
This is why the standardization and order-reduction steps below are logical
parts of the theorem rather than implementation details.

Triangles are especially useful because their three facet normals have a
unique positive linear dependence.  Besicovitch-type three-segment paths
govern important triangular regimes, and rigorous covering results are known
for classes of Besicovitch triangles \cite{CoultonMovshovich,Movshovich}.
Gibbs numerically explored Bellman's problem over the isosceles family and
reported candidate Moser quotients near $0.257882$ and $0.257856$, the latter
in a high-angle regime near $75.58^\circ$ \cite{Gibbs}.  Those computations are
valuable guides, but a numerical escape path is a primal upper bound on an
escape threshold; by itself it does not furnish the global lower bound on that
threshold needed for a Moser universal cover theorem.

Bellman's lost-in-a-forest problem supplies the relevant dual viewpoint.  For
a convex body $K$, let $E(K)$ be the infimum length of a rectifiable arc that
cannot be placed in $\interior K$ by a translation and a rotation.  If
$E(K)\ge D$, then the homothet $D^{-1}K$ covers every unit worm.  A general
transformed-boundary and discrete-optimization formulation was developed in
\cite{DengGeneral}; equivalence and convergence results and further variants
appear in \cite{DengProof}, and support-function specializations for strips and
triangles in \cite{DengStrip,DengTriangle}.

The balanced-support method of Temerev and Doria gives a particularly effective
way to prove Bellman lower bounds \cite{TemerevDoria}.  One integrates exact
support inequalities against a positive source measure, folds the resulting
masses through the positive dependence among the triangle normals, aggregates
along the normal fan, and applies a finite zero-sum vector ledger.  The
analytic part is a discrete Abel-summation estimate.  The geometric part is
independent and equally important: it must force every hypothetical shorter
minimizer into a temporal support order tolerated by the ledger.

In this paper, the high-angle triangle considered here is
\[
 T=\conv\{(-c,0),(c,0),(0,s)\},\qquad
 s=\frac{766}{\sqrt{625565}},\qquad
 c=\frac{197}{\sqrt{625565}},
\]
so that $\tan\beta=766/197$ and $\beta=75.5772223320\ldots^\circ$.  The exact
four-source finite certificate produces
\[
 D=\frac{82074390584}{84861020075}=0.9671624323094728\ldots,
 \qquad
 \frac{\area(T)}{D^2}
 =0.257883595112076188\ldots .
\]
The constants were discovered computationally, but after they are fixed the
finite ledger comparisons reduce to exact rational arithmetic and one rational
enclosure of $\sqrt{625565}$.

The present paper closes the former near-anchor branch gap rather than assuming
it.  After the continuum reduction and the selected $\Lambda$-gap
configuration are established, the corrected side-meeting analysis yields the
far-endpoint temporal inequalities.  If the corresponding near-anchor
inequality fails on either side, the delimiter-gap lemma produces an omitted
hull edge whose outward normal lies in the certified signed half-angle range.
The resulting cases are exhaustive: the inner branch is covered by the
original exact 25-order ledger, a single external near anchor is covered by the
one-delimiter signed interval certificate, and two external near anchors are
covered by the signed two-delimiter rectangle certificates together with the
exact reflection/time-reversal symmetry.  Consequently no near-anchor
hypothesis is used in the main theorem, and
\[
 E(T)\ge
 \frac{82074390584}{84861020075}
\]
holds unconditionally.  Bellman--Moser scaling therefore gives the convex
universal cover bound
\[
 M\le M_{\rm conv}\le
 \frac{11511821678274125}{44639604443512928}
 =0.257883595112076188\ldots
\]

The formal theorem proofs are supplied as a single Lean 4/Mathlib source in
\Cref{app:lean}.  To keep that source as a readable proof artifact, it contains
no generated interval/rectangle dataset.  It formalizes the metric deletion and
uncrossing lemmas used in planar standardization, the endpoint-peeling/deque
order core, the three-phase delimiter-order argument, the exact delimiter-angle
identity, the selected-$\Lambda$ implication from its explicit global geometric
inputs, and the scalar, radical, endpoint-ledger, fixed-delimiter-ledger,
target, and area identities.  The generated exact rational
signed-delimiter data have been removed from Lean and remain instead in the
supplementary JSON certificates, whose exhaustive rational replay is performed
by the independent Python verifier in \Cref{app:verifier}.  The Lean proof file
contains no declared \texttt{axiom}, \texttt{sorry}, \texttt{admit}, or
\texttt{native\_decide}.  For precision about formalization scope, the global
planar $\Lambda$-configuration theorem and the compactness/extreme-point
bookkeeping of the full continuum standardization theorem are not silently
postulated: where they are needed, they remain explicit hypotheses or ordinary
mathematical arguments.  Likewise, the exhaustive signed finite dataset is
not asserted as a theorem of the proof-only Lean source; it is a separate exact
certificate checked by the replay program.

The paper is organized as follows.  \Cref{sec:bellman} develops the
Bellman--Moser scaling principle, the exact triangular support criterion, and
the directed ledger.  \Cref{sec:geometry} gives the exact high-angle geometry,
continuum standardization, cyclic bitonicity, and the selected
$\Lambda$-gap.  \Cref{sec:metric} proves the corrected side-meeting and
one-sided exclusion estimates and derives the delimiter-gap mechanism from a
near-anchor failure.  \Cref{sec:finite} gives the four-source fold and the
exact 25-order inner certificate, while \Cref{sec:delimiter} gives the complete
signed one- and two-delimiter interval certificates.
\Cref{sec:conditional} states the unconditional Bellman and Moser theorems and
includes an exact geometric figure of the resulting homothet and a unit
V-worm.  \Cref{sec:branches} records the marked-$Z$ algebraic branch and the
triangle-restricted optimization framework.  \Cref{sec:duality} develops
balanced-circuit duality for general convex bodies and certified computational
hierarchies.  The final section records verification scope.  Exact norms, a
minimal order-theoretic audit, the consolidated Lean source, and the exact
symbolic verifier appear in the appendices.

\section{Bellman--Moser duality and support calibrations}\label{sec:bellman}

\subsection{Definitions and scaling}

\begin{definition}
For a nonempty compact set $H\subset\RR^2$ and an arbitrary vector
$v\in\RR^2$, define the positively homogeneous support function
\[
 h_H(v):=\max_{x\in H}x\cdot v.
\]
For $u_\phi=(\cos\phi,\sin\phi)$ write $h_H(\phi):=h_H(u_\phi)$.
\end{definition}

The extension from unit vectors to all $v\in\RR^2$ is deliberate: the folded
calibration vectors below carry positive masses, and then
$h_H(av)=a h_H(v)$ for $a\ge0$ is used without a change of notation.

\begin{definition}
Let $K\subset\RR^2$ be compact and convex with nonempty interior.  A
rectifiable arc $\Gamma$ is an \emph{escape path} for $K$ if no translated and
rotated copy of $\Gamma$ is contained in $\interior K$. Put
\[
 E(K):=\inf\{\len(\Gamma):\Gamma\text{ is an escape path for }K\}.
\]
\end{definition}

\begin{proposition}[Bellman--Moser scaling]\label[proposition]{prop:scaling}
If $E(K)\ge L>0$, then $L^{-1}K$ contains a congruent copy of every unit worm \cite{DengGeneral}. 
Consequently
\[
 M\le M_{\rm conv}\le \frac{\area(K)}{L^2}.
\]
\end{proposition}

\begin{proof}
Every arc of length $<L$ is placeable in $\interior K$ by the definition of
$E(K)$.  Let $\Gamma$ have length exactly $L$ and set
$\Gamma_n=(1-1/n)\Gamma$.  Choose
$z_n+R_{\theta_n}\Gamma_n\subset\interior K$.  After fixing one point of $\Gamma$,
compactness of $K$ bounds $z_n$; compactness of $SO(2)$ gives a subsequence
with $z_n\to z$ and $R_{\theta_n}\to R_\theta$.  Since $K$ is closed,
$z+R_\theta\Gamma\subset K$.  Thus $K$ covers every arc of length at most $L$.
Scaling a length-$L$ copy of a unit worm by $L^{-1}$ proves the claim, and area
scales by $L^{-2}$.
\end{proof}

\begin{proposition}[Quotient formulation]\label[proposition]{prop:quotient}
For compact convex bodies with nonempty interior,
\[
 M_{\rm conv}=\inf_K\frac{\area(K)}{E(K)^2}.
\]
\end{proposition}

\begin{proof}
The inequality ``$\le$'' follows from \Cref{prop:scaling}.  Conversely, let
$U$ be a convex universal unit-worm cover.  For $0<r<1$, a length-$r$ arc
$\Gamma$ becomes a unit worm after scaling by $r^{-1}$, hence has a placement
$P\subset U$.  Choose $y\in\interior U$ and $\delta>0$ with
$\overline B(y,\delta)\subset U$.  For every $p\in P$ and $0<r<1$,
convexity gives
\[
 B\bigl(y+r(p-y),(1-r)\delta\bigr)\subset U,
\]
so $y+r(P-y)\subset\interior U$.  If
$P=z+R(r^{-1}\Gamma)$, then
$y+r(P-y)=\bigl(y+r(z-y)\bigr)+R\Gamma$, hence this is a rigid placement of
$\Gamma$.  Therefore $E(U)\ge1$ and
$\inf_K\area(K)/E(K)^2\le\area(U)$.  Infimize over $U$.
\end{proof}

\subsection{Exact support criterion for a triangle}

For a unit-base triangle with base angles $\alpha,\beta>0$ and
$\alpha+\beta<\pi$, let its outward unit normals be $n_1,n_2,n_3$ in cyclic
order.  Their positive dependence can be normalized as 
\begin{equation}\label{eq:triangle-balance}
 \sin\beta\,n_1+\sin\alpha\,n_2+\sin(\alpha+\beta)\,n_3=0.
\end{equation}

\begin{theorem}[Exact triangular support criterion]\label[theorem]{thm:tri-support}
A compact convex set $H$ fails to admit a translated and rotated placement in
the interior of the unit-base triangle $T_{\alpha,\beta}$ if and only if, for
every $t\in\RR$ \cite{DengTriangle},
\[
 \sin\beta\,h_H(t+\pi+\alpha)+\sin\alpha\,h_H(t+\pi-\beta)
 +\sin(\alpha+\beta)h_H(t)\ge \sin\alpha\sin\beta.
\]
\end{theorem}

\begin{proof}
It is useful to make the strict-feasibility step completely explicit.  Place
the unit base at $(0,0)$ and $(1,0)$.  With the third vertex above the base,
one may take the three outward unit normals and offsets to be
\[
 n_1=(-\sin\alpha,\cos\alpha),\quad b_1=0,\qquad
 n_2=(\sin\beta,\cos\beta),\quad b_2=\sin\beta,
\]
\[
 n_3=(0,-1),\qquad b_3=0.
\]
Then
\[
 \lambda_1n_1+\lambda_2n_2+\lambda_3n_3=0,
 \qquad
 (\lambda_1,\lambda_2,\lambda_3)
 =(\sin\beta,\sin\alpha,\sin(\alpha+\beta)),
\]
and
\[
 \lambda_1b_1+\lambda_2b_2+\lambda_3b_3
 =\sin\alpha\sin\beta.
\]
For this fixed orientation, translating $H$ into the three open supporting
half-planes is equivalent to
\[
 n_i\cdot z<r_i:=b_i-h_H(n_i),\qquad i=1,2,3.
\]
Let $N:\RR^2\to\RR^3$ be $Nz=(n_1\cdot z,n_2\cdot z,n_3\cdot z)$ and put
$\lambda=(\lambda_1,\lambda_2,\lambda_3)$.  The normals span $\RR^2$, so
$N$ has rank two.  Since $\lambda$ spans the one-dimensional left kernel,
\[
 \operatorname{im}N=\{y\in\RR^3:\lambda\cdot y=0\}.
\]
If $Nz<r$ componentwise, positivity of the $\lambda_i$ gives
$0=\lambda\cdot Nz<\lambda\cdot r$.  Conversely, if
$q:=\lambda\cdot r>0$, set
\[
 \varepsilon:=\frac{q}{\lambda_1+\lambda_2+\lambda_3}>0,
 \qquad y:=r-\varepsilon(1,1,1).
\]
Then $\lambda\cdot y=0$, hence $y=Nz$ for some $z$, and $y_i<r_i$ for all
$i$.  Thus strict translation feasibility is equivalent exactly to
$\lambda\cdot r>0$.  Substituting the offsets above shows that failure of
strict placement at this orientation is equivalent to
\[
 \sin\beta\,h_H(n_1)+\sin\alpha\,h_H(n_2)
 +\sin(\alpha+\beta)h_H(n_3)\ge\sin\alpha\sin\beta.
\]
Rotating the three normals together yields the asserted angular formula.
\end{proof}

Fix $0<\beta<\pi/2$, set $s=\sin\beta$, $c=\cos\beta$, and define
\[
 T_\beta:=\conv\{(-c,0),(c,0),(0,s)\},\qquad \area(T_\beta)=sc.
\]
Then the equal sides have length one.

\begin{corollary}[Isosceles support criterion]\label[corollary]{cor:iso-support}
A compact convex set $H$ is an escape hull for $T_\beta$ if and only if
\begin{equation}\label{eq:iso-support}
 2c\,h_H(\phi)+h_H(\phi+\pi-\beta)+h_H(\phi+\pi+\beta)\ge2sc
 \qquad(\phi\in\RR).
\end{equation}
Moreover
\[
 2c\,u_\phi+u_{\phi+\pi-\beta}+u_{\phi+\pi+\beta}=0.
\]
\end{corollary}

Because $h_\Gamma=h_{\conv\Gamma}$, escape depends only on the convex hull of
the path.

\subsection{Balanced measures and the directed ledger}

Let $\nu$ be a finite positive measure on source angles and put
\[
 \mu_\nu:=2c\,\nu+(\tau_{\pi-\beta})_\#\nu+(\tau_{\pi+\beta})_\#\nu.
\]
Integrating \eqref{eq:iso-support} gives, for every escape hull $H$,
\begin{equation}\label{eq:integrated-support}
 \int h_H(u_\theta)\,d\mu_\nu(\theta)\ge2sc\,\nu(\SSS),
 \qquad
 \int u_\theta\,d\mu_\nu(\theta)=0.
\end{equation}

\begin{lemma}[Directed support ledger]\label[lemma]{lem:ledger}
Let $P_0,\dots,P_N$ be the vertices of a polygonal path in temporal order and
let $v_1,\dots,v_m\in\RR^2$ satisfy $\sum_jv_j=0$.  Choose for every $j$ a
path vertex $P_{\iota(j)}$ with
$P_{\iota(j)}\cdot v_j=h_H(v_j)$, where
$H=\conv\{P_0,\dots,P_N\}$.  Put
\[
 R_i:=\sum_{\iota(j)\ge i}v_j,\qquad i=1,\dots,N.
\]
If $\|R_i\|\le R$ for every $i$, then
\[
 \sum_{j=1}^m h_H(v_j)\le R\sum_{i=0}^{N-1}\|P_{i+1}-P_i\|.
\]
\end{lemma}

\begin{proof}
Write
$P_{\iota(j)}=P_0+\sum_{i<\iota(j)}(P_{i+1}-P_i)$.
The $P_0$ term cancels because $\sum_jv_j=0$.  Interchanging finite sums yields
\[
 \sum_jh_H(v_j)=\sum_{i=0}^{N-1}(P_{i+1}-P_i)\cdot R_{i+1}.
\]
Cauchy--Schwarz gives the result.
\end{proof}

If a direction exposes an edge, its mass may be assigned to either endpoint
or split between them.  Intermediate split states are convex combinations of
the two endpoint ledger states, so a Euclidean radius bound survives splitting.

\begin{lemma}[Reversal invariance for a balanced ledger]\label[lemma]{lem:reversal}
Let $v_1,\dots,v_m$ satisfy $\sum_jv_j=0$.  The multiset of norms of nonzero
suffix sums of $(v_m,\dots,v_1)$ equals the multiset of norms of nonzero prefix
sums of $(v_1,\dots,v_m)$, and each such prefix is the negative of a
complementary suffix.  Hence any uniform suffix-radius bound is invariant under
reversing the entire vector order.
\end{lemma}

\begin{proof}
For $1\le k<m$,
$\sum_{j=1}^kv_j=-\sum_{j=k+1}^mv_j$.  A suffix of the reversed sequence is a
prefix of the original sequence.  Taking norms proves the assertion.
\end{proof}

\section{High-angle geometry and continuum reduction}\label{sec:geometry}

\subsection{The exact triangle}

Put
\[
 p=766,\quad q=197,\quad S=p^2+q^2=625565,
 \qquad s=\frac p{\sqrt S},\quad c=\frac q{\sqrt S}.
\]
Thus $\tan\beta=766/197$ and
\[
 \area(T)=sc=\frac{150902}{625565}.
\]
For the metric anchoring argument use the congruent unit-side realization
\[
 \Delta=\conv\{(0,0),(1,0),(\cos\delta,\sin\delta)\},
 \qquad \delta=\pi-2\beta.
\]
Its relevant data are rational:
\begin{equation}\label{eq:ABC}
 \rho:=\sin\delta=\frac{301804}{625565},\qquad
 A:=\cot\delta=\frac{547947}{301804},\qquad
 b:=\cot\beta=\frac{197}{766},
\end{equation}
with
\begin{equation}\label{eq:ABC-rel}
 A+b=\frac1\rho,\qquad A>b>0,\qquad \rho>\frac5{12},
\end{equation}
and
\begin{equation}\label{eq:Delta-halfplanes}
 \Delta=\{(x,y): y\ge0,\ x\ge Ay,\ x+by\le1\}.
\end{equation}

Define
\[
 \Phi(H):=\min_{\phi\in\RR}
 \bigl(2c\,h_H(\phi)+h_H(\phi+\pi-\beta)+h_H(\phi+\pi+\beta)\bigr).
\]
Then $H$ is an escape hull iff $\Phi(H)\ge2sc$.
The map $\Phi$ is translation invariant, positively homogeneous, and satisfies
\[
 |\Phi(H)-\Phi(K)|\le(2c+2)d_H(H,K).
\]

\subsection{Standardization}

\begin{definition}
A polygonal path is \emph{standard} if it is simple and its nonrepeated
vertices are exactly the extreme vertices of its convex hull, each visited
once.  A hull edge not traversed by the path is a \emph{gap}.  A standard path
is called \emph{convex} if its links follow one hull-boundary chain between its
endpoints; otherwise it is \emph{nonconvex}.
\end{definition}

\begin{proposition}[Standard polygonal reduction]\label[proposition]{prop:standardize}
If a rectifiable escape path of length $<C$ exists, then there are an integer
$N$ and a standard polygonal escape path $\eta$ of length $<C$ that is
length-minimal among polygonal escape paths with at most $N$ segments.
\end{proposition}

\begin{proof}
Let $\gamma$ be a rectifiable escape path with $\len\gamma<C$.  Choose refining
partitions whose chordal interpolants $\gamma_n$ converge uniformly to
$\gamma$ and whose lengths converge to $\len\gamma$.  Then the convex hulls
converge in Hausdorff distance, hence
$\Phi(\conv\gamma_n)\to\Phi(\conv\gamma)\ge2sc$.
Set
\[
 r_n:=\max\left\{1,\frac{2sc}{\Phi(\conv\gamma_n)}\right\}.
\]
For all sufficiently large $n$ the denominator is positive,
$r_n\to1$, and $r_n\gamma_n$ is a polygonal escape path of length $<C$.
Fix one with at most $N$ segments and length $M<C$.

Translate its initial point to the origin and pad shorter paths by terminal
repetitions.  The set of $(N+1)$-tuples with total length at most $M$, first
vertex $0$, and $\Phi\ge2sc$ is compact: all vertices lie in the radius-$M$
closed disk, length is continuous, and the finite-hull map is Hausdorff
continuous.  Thus length attains a minimum.  Among minimizers choose one with
the smallest number of temporal vertex occurrences after consecutive equal
entries are suppressed.

If one point occurs at two distinct temporal indices, delete either one of
those occurrences and join its two temporal neighbours directly.  The point
remains in the vertex set at its other occurrence, so the convex hull and the
escape constraint are unchanged; the triangle inequality does not increase
length.  This contradicts the secondary choice.  If a retained point is not
an extreme point of the hull, delete its unique occurrence instead.  Its
deletion again leaves the hull unchanged and cannot increase length.  Thus the
remaining temporal points are distinct and are exactly the extreme points of
their hull.

No three distinct extreme points of a planar polygon are collinear.  Hence an
intersection between nonadjacent links of the resulting path, if present, is
a proper crossing.  For links $AB$ and $CD$ occurring in that temporal order,
reverse the intervening block, replacing the two crossed links by $AC$ and
$BD$.  The vertex set and hull are unchanged.  If $X$ is the crossing point,
then
\[
 |AC|+|BD|<(|AX|+|XC|)+(|BX|+|XD|)=|AB|+|CD|,
\]
where strictness follows because no three hull vertices are collinear.  This
contradicts length minimality.  Thus the minimizer is simple and standard.
\end{proof}

\begin{proposition}[Cyclic bitonicity]\label[proposition]{prop:bitonic}
For a standard path through the vertices of a convex polygon, temporal ranks
around the cyclic hull boundary are cyclically bitonic.  If a gap $FT$ is cut
with $F\precT T$, then along either fixed orientation of the complementary hull
boundary the rank sequence has at most three monotone phases; after choosing
the orientation consistently it can be written as decreasing--increasing--decreasing.
\end{proposition}

\begin{proof}
For a simple Hamiltonian path through points in convex position, after deleting
the already visited vertices the unvisited vertices form one cyclic interval.
The next path vertex must be one of the two endpoints of that interval;
otherwise the first new chord separates remaining vertices on both sides and a
later path link crosses it.  Thus the temporal order is a deque order.  The two
endpoint-deletion streams are monotone in opposite cyclic directions and meet
at the final vertex, which is precisely cyclic bitonicity.  Cutting at a gap
produces the stated three-phase form once the boundary orientation is fixed.
\end{proof}

\subsection{Selected Lambda gap}

We use the Lambda-configuration theorem for simple arcs
\cite{AlexanderWetzelWichiramala,MovshovichWetzelFrames} and the selected-gap
surgery developed in \cite[Appendix A]{TemerevDoria}.  The hypotheses needed
for the surgery are recorded explicitly here.

\begin{proposition}[Selected $\Lambda$-gap]\label[proposition]{prop:lambda-gap}
Let $\eta$ be a nonconvex standard polygonal escape path that is length-minimal
among escape paths with at most $N$ segments.  Then there exist a gap $FT$, an
opposite hull vertex $M$, and two parallel support lines at distance $h>0$
such that, after a rigid normalization,
\[
 F=(x_F,0),\quad T=(x_T,0),\quad M=(x_M,h),\qquad
 x_F\le x_T,\qquad F\precT M\precT T,
\]
and
\begin{equation}\label{eq:lambda-height}
 \len(\eta)\ge(1+\sqrt2)h.
\end{equation}
\end{proposition}

\begin{proof}
The standard nonconvex path is a simple open polygonal arc of positive width.
Apply the $\Lambda$-configuration theorem \cite[Theorem A.6]{TemerevDoria}
and the two-gap estimate \cite[Lemma A.5]{TemerevDoria}; the same selected-gap
assembly is recorded in \cite[Lemma A.7]{TemerevDoria}.  The two-gap estimate
applies to a segment-count minimizer in any admissible class that is monotone
under convex-hull enlargement; the escape class has exactly this monotonicity by
\eqref{eq:iso-support}.  They provide two distinct contacts on one support
line and an intermediate temporal contact on the opposite support line.  Standardness implies that no path segment lies in the first support
line: otherwise that segment would be the exposed hull edge and the entire
temporal subarc between the two contacts would remain on the line, contrary to
the opposite contact.  Hence the exposed face is an omitted hull edge $FT$,
and after relabelling $F\precT M\precT T$.

Finally, reflect in a vertical line if necessary so that the spatial gap
orientation is $x_F\le x_T$.  If that reflection reverses the desired temporal
naming, reverse the path parameter and exchange the endpoint names.  Both
operations preserve length, standardness, escape, and the selected
$\Lambda$-configuration.  Thus the two displayed conventions hold
simultaneously.

At least one of $F,T$ is not a temporal endpoint of the path.  Indeed, suppose
both were endpoints.  Since $FT$ is a hull edge, $F$ and $T$ are adjacent hull
vertices.  Relabel temporal orientation so that $F$ is first and $T$ is last.
After deleting $F$, the unvisited vertices form one cyclic interval by
\Cref{prop:bitonic}; one endpoint of that interval is $T$.  Because $T$ must
remain unvisited until the final step, the next vertex is forced to be the
other endpoint.  Repeating this argument inductively forces the path to follow
the entire complementary hull-boundary chain from $F$ to $T$, contradicting
nonconvexity.  Thus the endpoint hypothesis required by the two-gap outer-cap
shortening lemma is satisfied.

The outer-cap replacement used in the cited two-gap lemma remains an escape
path because enlarging the convex hull only increases every support value in
\eqref{eq:iso-support}.  Segment-count minimality therefore applies exactly as
required, and \cite[Lemma A.5]{TemerevDoria} gives
\eqref{eq:lambda-height}.
\end{proof}

\section{Corrected high-angle metric estimate}\label{sec:metric}

Put
\begin{equation}\label{eq:D}
 D:=\frac{82074390584}{84861020075}=0.9671624323094728\ldots<1.
\end{equation}
Suppose henceforth that a nonconvex standard minimizer has length $L<D$ and
choose the gap of \Cref{prop:lambda-gap}.  Since $\sqrt2>7/5$,
\begin{equation}\label{eq:hbound}
 h<\frac{D}{1+\sqrt2}<\frac5{12}<\rho.
\end{equation}
Place the side $y=0$ of \eqref{eq:Delta-halfplanes} on the gap and translate
horizontally until $x=Ay$ supports the hull.  Choose any path vertex $B$ in
that exposed support face and write
\[
 B=(Ay_B,y_B),\qquad 0\le y_B\le h.
\]
(If the exposed face is an edge, either endpoint may be chosen.)

\begin{lemma}[The opposite side is met]\label[lemma]{lem:D0-meeting}
There is a point $D_0$ of the path (not necessarily a support point) satisfying
\[
 D_0=(1-by_D,y_D),\qquad 0\le y_D\le h.
\]
\end{lemma}

\begin{proof}
Define the affine clearance
$r(x,y):=1-x-by$.  At $B$, using \eqref{eq:ABC-rel},
\[
 r(B)=1-(A+b)y_B=1-\frac{y_B}{\rho}>0
\]
by \eqref{eq:hbound}.  If $r>0$ on the entire compact path, then
$\eps:=\min r>0$.  Choose $\Delta y>0$ so small that
$(A+1+b)\Delta y<\eps$ and set $\Delta x=(A+1)\Delta y$.
After translating the path by $(\Delta x,\Delta y)$ one has
\[
 y+\Delta y>0,
\quad (x+\Delta x)-A(y+\Delta y)=(x-Ay)+\Delta y>0,
\]
and
\[
 1-(x+\Delta x)-b(y+\Delta y)
 =r(x,y)-(A+1+b)\Delta y>0.
\]
Thus the translated path lies in $\interior\Delta$, contradicting escape.
Therefore some path point has $r\le0$.  Connect that point to $B$ along the
path; continuity of $r$ and $r(B)>0$ gives a point $D_0$ with $r(D_0)=0$.
The strip support gives $0\le y_D\le h$.
\end{proof}

The distinction in \Cref{lem:D0-meeting} is essential: $B$ is a genuine
support contact, whereas $D_0$ need only be a crossing of the third side.

\begin{lemma}[Chord-dual estimate]\label[lemma]{lem:chord-dual}
If $X_0\precT X_1\precT\cdots\precT X_m$ are points of a rectifiable path and
$\|q_i\|\le1$, then
\[
 \len(\Gamma)\ge\sum_{i=0}^{m-1}q_i\cdot(X_{i+1}-X_i).
\]
\end{lemma}

\begin{proof}
The temporal subarcs $X_i\rightsquigarrow X_{i+1}$ are disjoint, and each has
length at least $\|X_{i+1}-X_i\|\ge q_i\cdot(X_{i+1}-X_i)$.  Sum.
\end{proof}

\begin{proposition}[One-sided high-angle exclusion]\label[proposition]{prop:one-sided}
With the metric normalization above, every path-vertex contact $B$ of the
supporting face $x=Ay$ satisfies $B\precT T$.
\end{proposition}

\begin{proof}
Assume $T\preceqT B$.  Since $F\precT M\precT T$, inserting $D_0$ gives, up
to weak coincidences, exactly five orders.

For
\[
 F-M-D_0-T-B,\qquad F-D_0-M-T-B,
\]
put
\[
 X=\frac{483585}{10^6},\quad P=\frac{875291}{10^6},\quad
 Q=\frac{38347525561}{150902000000}=\frac{2X}{\rho}-2P.
\]
Use respectively
\[
 q_0=(X,P),\quad q_1=(X,2bX-P),\quad q_2=(-X,-P),\quad q_3=(-2X,Q),
\]
and in the second order replace $q_1$ by $(-X,P-2bX)$.  Exact rational
arithmetic gives
\[
 X^2+P^2<1,\quad X^2+(2bX-P)^2<1,\quad (2X)^2+Q^2<1,
\quad P-bX>0.
\]
Expanding \Cref{lem:chord-dual} gives in either case
\[
 L\ge2X+X(x_T-x_F)+2(P-bX)(h-y_B)\ge2X>D,
\]
where
\[
 2X-D=\frac{2568807751}{339444080300000}>0.
\]

For the third order $D_0-F-M-T-B$, put $a=9672/10000$ and use
\[
 (-a,-ba),\quad(0,1),\quad(0,-1),\quad(-a,ba).
\]
The only nontrivial norm/sign checks are
\[
 a^2(1+b^2)<1,\qquad 2+a(b-A)>0.
\]
The expansion is
\[
 L\ge a+a(x_T-x_F)+2h+a(b-A)y_B.
\]
Since $b-A<0$ and $y_B\le h$,
\[
 L\ge a+a(x_T-x_F)+[2+a(b-A)]h\ge a>D,
\]
with
\[
 a-D=\frac{159401627}{4243051003750}>0.
\]

The last two orders are
$F-M-T-D_0-B$ and $F-M-T-B-D_0$.  The disjoint subchain through $F,M,T$ has
length at least $2h$, while
\[
 x_{D_0}-x_B\ge1-(A+b)h=1-\frac h\rho>0.
\]
Hence $|BD_0|\ge1-h/\rho$ and
\[
 L\ge1-\left(\frac1\rho-2\right)h.
\]
Under $L<D$ and \eqref{eq:hbound},
$h<(5/12)D$.  Since
\[
 \frac5{12}\left(\frac1\rho-2\right)=\frac{36595}{1207216},
\]
we obtain
\[
 L>1-\frac{36595}{1207216}D>D,
\]
where the final margin is
\[
 1-\frac{36595}{1207216}D-D
 =\frac{597327047}{169722040150}>0.
\]
Every possible location of $D_0$ is covered, a contradiction.
\end{proof}

\subsection{Far anchors, inner anchors, and delimiter gaps}
\label{subsec:near-anchor}

The one-sided estimate above is attached to a definite calibration direction.
Put
\[
 \kappa:=s^2-c^2>0,\qquad
 Z_2:=\zeta_L:=(-\rho,\kappa),\qquad
 Z_1:=\zeta_R:=(\rho,\kappa).
\]
Since $A=\kappa/\rho$ and $\rho^2+\kappa^2=1$, the normalized outward normal
of the supporting side $x=Ay$ is
\[
 \frac{(-1,A)}{\sqrt{1+A^2}}=(-\rho,\kappa)=Z_2.
\]
Thus \Cref{prop:one-sided} proves the far-endpoint inequality for the left
anchor direction.

\begin{corollary}[The two far-endpoint inequalities]
\label[corollary]{cor:reflected-far}
For the selected gap one may choose contacts $Z_2,Z_1$ of the two anchor faces
such that
\[
 Z_2\precT T,\qquad F\precT Z_1.
\]
\end{corollary}

\begin{proof}
The first inequality is \Cref{prop:one-sided}.  For the second, let
$\mathcal R(x,y)=(-x,y)$, reflect the normalized hull, and reverse the temporal
parameter.  The transformed path is a standard escape minimizer with the same
length and selected gap
\[
 F'=\mathcal R(T),\qquad T'=\mathcal R(F),
\]
with $F'\prec_{\rm rev}M'\prec_{\rm rev}T'$ and $x_{F'}\le x_{T'}$.
Apply \Cref{prop:one-sided} to the transformed $Z_2$-face.  Reflection sends
the original $Z_1$-face to that face, and temporal reversal reverses the
inequality, giving $F\precT Z_1$.
\end{proof}

\begin{lemma}[Weak fan order of the anchor faces]
\label[lemma]{lem:weak-fan-order}
Orient the complementary hull boundary clockwise from $F$ through the upper
support face to $T$.  The contacts may be selected so that
\[
 F\le_{\rm fan} Z_2\le_{\rm fan} M\le_{\rm fan} Z_1\le_{\rm fan} T.
\]
Equalities are allowed when two directions belong to one vertex normal cone.
\end{lemma}

\begin{proof}
The selected gap has outward normal $d=(0,-1)$.  Along the clockwise normal
sweep its relevant directions occur as
\[
 d,\quad Z_2,\quad (0,1),\quad Z_1,\quad d.
\]
Exposed faces of a convex polygon vary monotonically with outer-normal angle.
Inside a vertex normal cone the contact may stay at the same vertex; between
cones it runs along the intervening edge.  Choosing a vertex from each face
gives the weak order.
\end{proof}

\begin{remark}[Former near-anchor hypothesis; now discharged]
\label[remark]{hyp:anchor-closure}
The two complementary relations
\begin{equation}\label{eq:anchor-time}
 F\precT Z_2,\qquad Z_1\precT T
\end{equation}
do not follow from cyclic bitonicity.  For example, on the five-point fan
$F<Z_2<M<Z_1<T$, ranks $(1,0,2,3,4)$ have the required bitonic form and obey
the far relations, but fail the first relation in \eqref{eq:anchor-time}.
The present proof therefore uses the old 25-order ledger only when both
relations hold and certifies every failure by a delimiter ledger.
\end{remark}

\begin{lemma}[Inner-anchor sweep]\label[lemma]{lem:anchored-sweep}
If both relations in \eqref{eq:anchor-time} hold, the two anchors lie on the
central increasing phase of the clockwise fan sweep.  Every temporal disorder
then consists of one fan prefix moved before $F$ and one fan suffix moved after
$T$.
\end{lemma}

\begin{proof}
Let $L_0,H_0$ be the phase-change contacts of the
decreasing--increasing--decreasing sweep from \Cref{prop:bitonic}.  If
$Z_2\le_{\rm fan}L_0$, monotonicity on the first phase gives
$t(Z_2)\le t(F)$, contradicting $F\precT Z_2$.  If
$H_0\le_{\rm fan}Z_1$, the final phase gives $t(T)\le t(Z_1)$,
contradicting $Z_1\precT T$.  Hence
\[
 L_0<_{\rm fan}Z_2\le_{\rm fan}M\le_{\rm fan}Z_1<_{\rm fan}H_0.
\]
Before $Z_2$, the condition $t(X)\le t(F)$ can change from true to false at
most once, so exactly one initial fan prefix may precede $F$.  The time-reversed
argument gives one terminal suffix after $T$.  The remaining contacts occur in
increasing fan order.
\end{proof}

\begin{lemma}[Right delimiter gap]\label[lemma]{lem:delimiter-gap}
Let $M=V_0,V_1,\ldots,V_n=Z_1$ be the consecutive hull vertices on the
clockwise boundary arc from an upward contact to a right-anchor contact.  If
\[
 M\precT T\precT Z_1,
\]
then some adjacent pair $U=V_{j-1},V=V_j$ satisfies
\[
 F\precT U\precT T\precT V.
\]
The edge $UV$ is a gap.  Its outward normal lies in the closed normal-fan
interval from $(0,1)$ to $Z_1$, so the acute angle of $UV$ with the selected
horizontal gap is at most
\[
 \delta=\pi-2\beta,
 \qquad \cos\delta=\kappa=\frac{547947}{625565}.
\]
\end{lemma}

\begin{proof}
Color $V_i$ blue if $V_i\precT T$ and red if $T\precT V_i$.  The first vertex
is blue and the last red.  Let $j$ be the first red index.  Then
$U\precT T\precT V$.  Suppose $U\precT F$.  Immediately after $F$ is visited,
both $F,U$ have been visited, while $M,V$ remain unvisited because
$F\precT M\precT T\precT V$.  In cyclic hull order the four occur as
\[
 F,\ldots,M,\ldots,U,V,\ldots,F.
\]
Thus the visited vertices $F,U$ separate the unvisited vertices $M,V$ into two
components, contradicting the deque property proved in
\Cref{prop:bitonic}.  Hence $F\precT U$.

The vertices $U,V$ are adjacent on the hull.  They cannot be consecutive path
vertices because $T$ lies strictly between them in time, so $UV$ is omitted.
Outer normals rotate monotonically along the convex boundary; therefore the
normal of $UV$ lies between those of the endpoint support faces, namely
$(0,1)$ and $Z_1$.  Rotating normals by $\pi/2$ gives the asserted tangent-angle
bound, and $\cos(\pi-2\beta)=s^2-c^2=\kappa$.
\end{proof}

\begin{corollary}[Left delimiter gap]\label[corollary]{cor:left-delimiter}
If $Z_2\precT F\precT M$, then there is a left delimiter gap $WX$ whose
normal lies between $Z_2$ and $(0,1)$ and whose endpoints satisfy
\[
 W\precT F\precT X\precT T.
\]
\end{corollary}

\begin{proof}
Apply \Cref{lem:delimiter-gap} after reflection in $x=0$ and temporal reversal.
Under this transformation the right delimiter endpoints return in reverse
order, giving precisely $W\precT F\precT X\precT T$.
\end{proof}

\section{Exact four-source finite certificate}\label{sec:finite}

Put
\[
 \eta=\frac\pi2-\beta,\qquad \mu=\frac{10973}{125000},\qquad
 \nu=\delta_\eta+\delta_{\pi-\eta}+\mu\delta_{2\eta}+\mu\delta_{\pi-2\eta}.
\]
Then $\nu(\SSS)=2+2\mu$, and \eqref{eq:integrated-support} gives the exact mass
\begin{equation}\label{eq:Bmass}
 B_*:=2sc\,\nu(\SSS)=4sc(1+\mu)
 =\frac{10259298823}{9774453125}.
\end{equation}

Cut the folded fan at $d=(0,-1)$ and define
\begin{align*}
 P_1&=\mu(c,-s),&
 P_2&=2c(s,c),\\
 P_3&=2c\mu(s^2-c^2,2sc),&
 P_4&=\mu(3c-4c^3,4s^3-3s),\\
 Z_1&=(2sc,s^2-c^2)=\zeta_R,&
 Z_2&=(-2sc,s^2-c^2)=\zeta_L,\\
 Q_1&=\mu(-(3c-4c^3),4s^3-3s),&
 Q_2&=2c\mu(-(s^2-c^2),2sc),\\
 Q_3&=2c(-s,c),&
 Q_4&=\mu(-c,-s).
\end{align*}
Let $d_1=d_2=d$.

The fold may be audited source by source:
\begin{center}
\begin{tabular}{ccl}
\toprule
source & weight & three folded vectors\\
\midrule
$\eta$ & $1$ & $P_2,\ Z_2,\ d_1$\\
$\pi-\eta$ & $1$ & $Q_3,\ d_2,\ Z_1$\\
$2\eta$ & $\mu$ & $P_3,\ Q_1,\ P_1$\\
$\pi-2\eta$ & $\mu$ & $Q_2,\ Q_4,\ P_4$\\
\bottomrule
\end{tabular}
\end{center}
Each row sums to zero by the three-normal identity, so
\begin{equation}\label{eq:balance-family}
 P_1+P_2+P_3+P_4+Z_1+Z_2+Q_1+Q_2+Q_3+Q_4+d_1+d_2=0.
\end{equation}
Exact cross products are positive in the cyclic vector order
\begin{equation}\label{eq:cyclic-vectors}
 d,P_1,P_2,P_3,P_4,Z_1,Z_2,Q_1,Q_2,Q_3,Q_4,d.
\end{equation}
For auditability, no numerical angle sorting is required.  The consecutive
cross products in \eqref{eq:cyclic-vectors} are, in order,
\[
\begin{gathered}
 \frac{2161681\sqrt{625565}}{78195625000},\quad
 \frac{2161681\sqrt{625565}}{39097812500},\quad
 \frac{83892677929\sqrt{625565}}{12229111538281250},\\
 \frac{1127417054285365682771879}{956263019930613769531250000},\quad
 \frac{2161681\sqrt{625565}}{78195625000},\quad
 \frac{330745192776}{391331569225},\\
 \frac{2161681\sqrt{625565}}{78195625000},\quad
 \frac{1127417054285365682771879}{956263019930613769531250000},\quad
 \frac{83892677929\sqrt{625565}}{12229111538281250},\\
 \frac{2161681\sqrt{625565}}{39097812500},\quad
 \frac{2161681\sqrt{625565}}{78195625000}.
\end{gathered}
\]
Every quantity is strictly positive; hence the displayed order is exact.

\subsection{The 25 certified orders}

The vector names in \eqref{eq:cyclic-vectors} use the counterclockwise normal
fan: $Z_1=\zeta_R$ and $Z_2=\zeta_L$.  For
$k,\ell\in\{0,1,2,3,4\}$ define
\begin{equation}\label{eq:orders}
 O_{k,\ell}:=(P_k,\dots,P_1),d_1,(P_{k+1},\dots,P_4),Z_1,Z_2,
 (Q_1,\dots,Q_{4-\ell}),d_2,(Q_4,\dots,Q_{5-\ell}),
\end{equation}
with empty ranges omitted.

\begin{lemma}[Inner-branch fan-to-ledger interface]\label[lemma]{lem:fan-ledger-interface}
If a nonconvex standard minimizer of length $<D$ satisfies both inner-anchor
relations \eqref{eq:anchor-time}, then, after possibly reversing the entire
temporal parametrization, the folded support vectors occur in one of the 25
orders \eqref{eq:orders}.
\end{lemma}

\begin{proof}
Along the clockwise complementary boundary from $F$ to $T$, the reverse of
\eqref{eq:cyclic-vectors} is
\[
 d,\ Q_4,Q_3,Q_2,Q_1,\ Z_2,Z_1,\ P_4,P_3,P_2,P_1,\ d.
\]
By \Cref{lem:anchored-sweep}, an initial block among
$Q_4,Q_3,Q_2,Q_1$ may migrate before $F$, and a terminal block among
$P_4,P_3,P_2,P_1$ may migrate after $T$; the remaining blocks occur in the
clockwise fan order between $F$ and $T$.  If a calibration direction exposes
an edge, assign its mass to an endpoint consistent with that weak fan order;
any split assignment is controlled by the convexity observation following
\Cref{lem:ledger}.

Reverse the entire temporal order.  If the terminal $P$ block has size $k$
and the initial $Q$ block has size $\ell$, the reversed sequence is exactly
$O_{k,\ell}$: $d_1$ is the right gap endpoint $T$ and $d_2$ the left endpoint
$F$ in that reversed convention.  A global reversal is harmless for the
radius estimate by \Cref{lem:reversal}.
\end{proof}

\begin{lemma}[Exact ledger certificate]\label[lemma]{lem:exact-ledger}
Every nonzero suffix state $R$ of every order $O_{k,\ell}$ satisfies
\[
 \|R\|<\Rhat,\qquad \Rhat:=\frac{27131}{25000}=1.08524.
\]
\end{lemma}

\begin{proof}
There are 11 nonzero suffixes per order, hence 275 raw states.  Exact symbolic
reconstruction reduces them to the 13 squared norms in \Cref{tab:norms}.
Every radical expression has the form $a\sqrt{625565}+b$ with $a\ge0$ rational.
Use the single exact enclosure
\[
 \sqrt{625565}<\frac{790927}{1000},\qquad
 \left(\frac{790927}{1000}\right)^2-625565
 =\frac{519329}{10^6}>0.
\]
Substitution into all 13 expressions gives a rational value strictly below
$\Rhat^2$.  The smallest squared-radius margin is
\[
 \frac{2404020037}{488722656250000}>0.
\]
The accompanying symbolic verifier reconstructs the vector family, the 25
orders, all 275 suffix states, their 13 distinct values, and these comparisons
from the definitions rather than loading a precomputed table.
\end{proof}

\section{Exact signed delimiter certificates}\label{sec:delimiter}

The preceding delimiter lemma converts every genuine near-anchor failure into
an exposed gap with a controlled normal.  This section gives a finite exact
ledger over the entire normal interval.  Put
\[
 \chi:=2c^2=\frac{77618}{625565}
\]
and define the four auxiliary fixed atoms
\[
 E_1=(\rho,-\chi),\quad E_2=(-\rho,-\kappa),\quad
 E_3=(-\rho,-\chi),\quad E_4=(\rho,-\kappa).
\]
Besides the two original rows containing $P_2,Z_2,Q_3,Z_1$, apply the exact
escape criterion at $-\eta$ and $\pi+\eta$.  The four fixed balanced rows are
\begin{equation}\label{eq:fixed-delimiter-rows}
\begin{array}{c|ccc}
\text{source}&\multicolumn{3}{c}{\text{balanced atoms}}\\ \hline
\eta&P_2&Z_2&d\\
\pi-\eta&Q_3&d&Z_1\\
-\eta&E_1&m&E_2\\
\pi+\eta&E_3&E_4&m
\end{array}
\qquad m=(0,1),\quad d=(0,-1).
\end{equation}
Every row sums to zero and its support sum is at least $\rho$.

\subsection{The complete signed normal parameter}

Let $n_R$ be the outward normal of a right delimiter.  By
\Cref{lem:delimiter-gap}, its polar angle belongs to
\[
 [\,2\beta-\tfrac\pi2,\tfrac\pi2\,]
 =[\,\beta-\eta,\beta+\eta\,].
\]
Thus write $n_R=u_{\beta+\theta_R}$ with
$-\eta\le\theta_R\le\eta$ and put
\[
 t_R=\tan\frac{\theta_R}{2},\qquad
 C_R=\frac{1-t_R^2}{1+t_R^2},\qquad
 S_R=\frac{2t_R}{1+t_R^2}.
\]
The source $\phi_R=\pi+\theta_R$ has the three balanced atoms
\begin{align}
 A_R&=\frac{(-2qC_R,-2qS_R)}{\sqrt S},\nonumber\\
 B_R&=\frac{(qC_R+pS_R,qS_R-pC_R)}{\sqrt S},\label{eq:right-moving-row}\\
 N_R&=\frac{(qC_R-pS_R,pC_R+qS_R)}{\sqrt S}=n_R.\nonumber
\end{align}
For a left delimiter write its normal as
$N_L=u_{\pi-\beta-\theta_L}$ with the same signed range and use the balanced
row
\begin{align}
 A_L&=\frac{(2qC_L,-2qS_L)}{\sqrt S},\nonumber\\
 N_L&=\frac{(-qC_L+pS_L,pC_L+qS_L)}{\sqrt S},\label{eq:left-moving-row}\\
 B_L&=\frac{(-qC_L-pS_L,-pC_L+qS_L)}{\sqrt S}.\nonumber
\end{align}
The signs are essential.  The earlier quarter-square calculation
$t_R,t_L\ge0$ covers only half of each permitted fan arc; the certificate
below uses the full signed range.

The exact radical enclosure
\begin{equation}\label{eq:two-sided-root}
 \frac{790926}{1000}<\sqrt{625565}<\frac{790927}{1000}
\end{equation}
follows by squaring, since the two positive margins are respectively
$265631/250000$ and $519329/10^6$.  Moreover
\[
 |t_R|,|t_L|\le t_*:=\tan\frac\eta2
 =\frac{197}{\sqrt{625565}+766}
 <\widehat t:=\frac{197000}{1556926}.
\]

\subsection{Exact interval-ledger principle}

For a rational interval $I=[a,b]\subset[-\widehat t,\widehat t]$, the function
$S(t)=2t/(1+t^2)$ is increasing.  The function
$C(t)=(1-t^2)/(1+t^2)$ increases on the negative half and decreases on the
positive half; if $0\in I$, its exact upper endpoint is $1$.  These facts give
rational coordinate intervals for \eqref{eq:right-moving-row} and
\eqref{eq:left-moving-row}.  Multiplication by the outward interval
\[
 \frac{1000}{790927}<\frac1{\sqrt S}<\frac{1000}{790926}
\]
then gives rational boxes containing every moving atom.

Each stored calibration consists of nonnegative rational source weights whose
sum is exactly
\[
 \frac{R}{\rho},\qquad R:=\frac{96717}{100000},
\]
together with rational divisions of the aggregate $d$-mass between $F,T$ and
of each delimiter-normal mass between the endpoints of its exposed edge.
Because both endpoints have the same support value, these divisions preserve
the support sum and exact balance.  The right-hand sides of the weighted
source inequalities sum to $R$.

For every temporal order, form every nonzero suffix of the assigned vector
masses.  Outward interval addition and squaring give a rational upper bound on
its squared norm.  If all these upper bounds are below one, the directed ledger
with radius one yields
\[
 R\le\sum_jh_H(v_j)\le\len(\gamma).
\]
This argument is uniform over the whole parameter cell and contains no
floating-point inference.

\begin{theorem}[Exact one-delimiter ledger]\label[theorem]{thm:one-delimiter}
Assume a genuine right failure and no genuine left failure, so after resolving
coincident labels at one temporal vertex,
\[
 F\precT M\precT T\precT Z_1,\qquad
 F\precT Z_2\precT T,
\]
and let the delimiter satisfy $F\precT U\precT T\precT V$.  For every
$t_R\in[-t_*,t_*]$ and every compatible weak deque order of the marked
contacts,
\[
 \len(\gamma)\ge R>D.
\]
\end{theorem}

\begin{proof}
In clockwise fan order the mass-bearing labels are
\[
 F,E_2,E_3,A_R,Q_3,Z_2,M,U,V,Z_1,P_2,E_1,E_4,B_R,T.
\]
The deque rule and the displayed temporal incidences generate exactly 512
strict marked orders.  A weak coincidence is covered by any strict refinement
because all labels at that temporal vertex have the same spatial contact.

The rational certificate partitions $[-\widehat t,\widehat t]$ into 1540
closed intervals with disjoint interiors: seven on the negative half and 1533
on the positive half.  On each interval, for each of the 512 orders, at least
one stored calibration has every squared suffix norm below one.  The global
exact upper bound $W_1$ satisfies
\[
 W_1<\frac{99999999}{100000000}<1.
\]
The interval-ledger principle gives $\len(\gamma)\ge R$.  Finally,
\[
 R-D=\frac{2568807751}{339444080300000}>0.
\]
The replay program reconstructs the orders and verifies every assertion in
this paragraph from the rational files.
\end{proof}

\begin{theorem}[Exact independent two-delimiter ledger]
\label[theorem]{thm:two-delimiter}
Assume both near anchors genuinely fail:
\[
 Z_2\precT F\precT M\precT T\precT Z_1.
\]
Let the right and left delimiter gaps satisfy
\[
 F\precT U\precT T\precT V,
 \qquad W\precT F\precT X\precT T.
\]
For every independent pair $(t_R,t_L)\in[-t_*,t_*]^2$ and every compatible
weak deque order of all marked contacts,
\[
 \len(\gamma)\ge R>D.
\]
\end{theorem}

\begin{proof}
The mass-bearing fan is
\[
 F,B_L,E_2,E_3,A_R,Q_3,Z_2,W,X,M,U,V,Z_1,P_2,A_L,E_1,E_4,B_R,T.
\]
The deque rule and the displayed incidences generate exactly 4096 strict
orders.  No equality between $t_R$ and $t_L$ is imposed.

The positive--positive quadrant is tiled by 736 rational rectangles, the
negative--negative quadrant by 256, and the negative--positive quadrant by
256.  Reflection in $x=0$ followed by temporal reversal interchanges
$t_R,t_L$ and bijects the 4096 orders, so the same 256 cells certify the
positive--negative quadrant.  Thus there are 1504 rectangle instances over
the full square.  This last transport is exact, not an interval re-evaluation:
for $\mathcal R(x,y)=(-x,y)$, direct substitution in
\eqref{eq:right-moving-row}--\eqref{eq:left-moving-row} gives
\[
 \mathcal R A_R(t)=A_L(t),\qquad
 \mathcal R B_R(t)=B_L(t),\qquad
 \mathcal R N_R(t)=N_L(t).
\]
The fixed rows are exchanged in pairs, and the rational calibration data
transform by
\[
 (w_0,w_1,w_2,w_3,w_4,w_5)
 \longmapsto(w_1,w_0,w_3,w_2,w_5,w_4),
\]
with each endpoint split sent to the reflected endpoint.  Reversal invariance
from \Cref{lem:reversal} then preserves every suffix norm exactly.

In each stored instance and for each order, one calibration has all squared
suffix norms below one; the reflected instances inherit the same bounds by
the displayed isometry.  The global exact upper bound satisfies
\[
 W_2<\frac{999938}{1000000}<1.
\]

For every stored quadrant the verifier checks containment, pairwise disjoint
rectangle interiors, and the exact area identity equalling the quadrant area.
A finite union of closed rectangles with full area and disjoint interiors
cannot omit a point: a nonempty relative-open complement would have positive
area.  Hence the cells cover the full signed square.  The interval-ledger
principle gives $\len(\gamma)\ge R>D$.
\end{proof}

\section{Unconditional Bellman and Moser theorems}\label{sec:conditional}

\begin{theorem}[Unconditional Bellman lower bound]
\label[theorem]{thm:unconditional-Bellman}
\label[theorem]{thm:conditional-Bellman}
The high-angle triangle $T$ satisfies
\[
 E(T)\ge D=\frac{82074390584}{84861020075}.
\]
\end{theorem}

\begin{proof}
Suppose an escape path has length $<D$.  By \Cref{prop:standardize}, choose a
standard polygonal escape minimizer $\gamma$ of length $L<D$ within a fixed
segment class.  If $\gamma$ is convex, its support contacts have one normal-fan
order (or its reversal), covered by $O_{0,0}$ and \Cref{lem:reversal}.  If it
is nonconvex, choose the selected gap and anchor contacts.  The far relations
$Z_2\precT T$ and $F\precT Z_1$ hold by \Cref{cor:reflected-far}.

If both inner relations \eqref{eq:anchor-time} hold,
\Cref{lem:fan-ledger-interface,lem:exact-ledger} and the four-source support
mass give
\[
 B_*\le\sum_jh_H(v_j)\le\Rhat L,
 \qquad
 L\ge\frac{B_*}{\Rhat}=D,
\]
a contradiction.

Otherwise reflect the configuration and reverse time if needed so that the
right relation genuinely fails, $T\precT Z_1$; coincident labels may be assigned
to the inner branch because their vector masses act at the same point.
\Cref{lem:delimiter-gap} supplies $UV$ with
$F\precT U\precT T\precT V$.  If $F\precT Z_2$, apply
\Cref{thm:one-delimiter}.  If $Z_2\precT F$, then
\Cref{cor:left-delimiter} also supplies $WX$ and
\Cref{thm:two-delimiter} applies.  Both alternatives give
$L\ge R>D$, again a contradiction.  These cases are exhaustive, so
$E(T)\ge D$.
\end{proof}

\begin{corollary}[Unconditional triangular universal cover]
\label[corollary]{cor:unconditional-cover}
\label[corollary]{cor:conditional-cover}
Then $D^{-1}T$ contains a congruent copy of
every rectifiable planar unit arc and
\[
 M\le M_{\rm conv}\le\area(D^{-1}T)
 =\frac{11511821678274125}{44639604443512928}
 =0.257883595112076188\ldots<0.257884.
\]
\end{corollary}

\begin{proof}
Apply \Cref{prop:scaling,thm:unconditional-Bellman}.  Since
$\area(T)=150902/625565$, exact rational reduction gives the displayed
fraction.  The strict comparison is certified by
\[
 \frac{64471}{250000}-
 \frac{11511821678274125}{44639604443512928}
 =\frac{282406824420693}{697493819429889500000}>0.
\]
\end{proof}

\medskip
For a concrete geometric visualization of the homothet in
\Cref{cor:conditional-cover}, set
\[
 A_0=(0,s/D),\qquad L_0=(-c/D,0),\qquad R_0=(c/D,0).
\]
The two equal sides of $D^{-1}T$ have length $1/D$.  

\begin{figure}[htbp]
    \centering
    \includegraphics[width=0.25\linewidth]{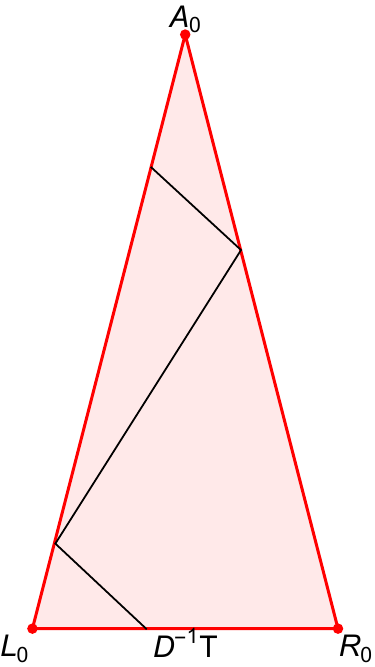}
\caption{The homothetic triangle $D^{-1}T$ appearing in
\Cref{cor:conditional-cover}.  The depicted worm is only an
illustrative feasible placement, not an asserted extremal escape path.  The
universal cover statement follows from the unconditional Bellman lower bound,
not from this one placement.}\label{fig:conditional-cover}
\end{figure}

\section{Marked-$Z$ branch and triangle-restricted optimization}\label{sec:branches}

\subsection{An ancillary exact marked-$Z$ algebraic value}

For comparison with the all-order calculation, define
\[
 t:=\sin\beta\cos\beta,
 \qquad
 Q_Z(\beta):=\frac{5-4\sqrt{1-9t^2}}{36t}.
\]
This is the closed algebraic value produced by the marked three-segment $Z$
ansatz; it is not used in any proof of the unconditional all-worm theorem.  At
the exact angle of this paper,
\[
 t=\frac{150902}{625565},\qquad
 1-9t^2=\frac{186388846789}{625565^2},
\]
so direct substitution gives
\[
 Q_Z=\frac{3127825-4\sqrt{186388846789}}{5432472}
 =0.25787781610863537955\ldots .
\]
The rational comparison is exact:
\[
 \sqrt{186388846789}>\frac{4317277}{10}
\]
because
\[
 186388846789-\left(\frac{4317277}{10}\right)^2
 =\frac{3984171}{100}>0,
\]
and therefore
\[
 Q_Z<\frac{128939}{500000}=0.257878.
\]
The Lean appendix checks precisely this radical comparison.

\begin{remark}[Scope of the marked-$Z$ value]
The displayed quantity is an algebraic branch value for the specified marked
ansatz.  This paper does not use it to assert $E(T)$, does not assert equality
with a triangle-restricted optimum, and does not claim an unrestricted
universal cover of area $Q_Z$.
\end{remark}

\subsection{Triangle-restricted optimization}

For a unit-base triangle with base angles $\alpha,\beta>0$ and
$\alpha+\beta<\pi$, write $E(\alpha,\beta)$ for its Bellman escape threshold.
By \Cref{prop:quotient}, the triangle-restricted convex quotient is
\[
 M_{\triangle}:=
 \inf_{\substack{\alpha,\beta>0\\\alpha+\beta<\pi}}
 \frac{\sin\alpha\sin\beta}
 {2\sin(\alpha+\beta)E(\alpha,\beta)^2}.
\]
The unconditional \Cref{cor:unconditional-cover} gives
\[
 M_{\triangle}\le0.257883595112076188\ldots .
\]
No equality statement is made.

A rigorous global optimization has asymmetric certification directions.  At a
prospective winner, a calibrated lower bound $D_0\le E(\alpha,\beta)$ yields the
upper quotient $A(\alpha,\beta)/D_0^2$.  On a competing angle box, an
independently continuum-certified escape path of length
$U\ge E(\alpha,\beta)$ yields the lower quotient
\[
 \frac{A(\alpha,\beta)}{E(\alpha,\beta)^2}
 \ge \frac{A(\alpha,\beta)}{U^2}.
\]
Consequently an interval branch-and-bound scheme can eliminate angle boxes
whose certified lower quotient exceeds the incumbent.  The difficult analytic
step in the high-angle isosceles family is not the scalar arithmetic displayed
here but the global certification of all competing temporal branches.
\section{Convex-body duality and certified computational hierarchies}\label{sec:duality}

\subsection{Balanced circuits for convex polygons}

Let
\[
 K=\bigcap_{j=1}^m\{x\in\RR^2:n_j\cdot x\le p_j\}
\]
be a compact convex polygon with outward unit normals $n_j$.  For a compact
convex set $H$ and fixed rotation $R_\theta$, a translation $z$ places
$R_\theta H$ in $K$ exactly when
\[
 n_j\cdot z\le p_j-h_H(R_{-\theta}n_j),\qquad j=1,\dots,m.
\]

\begin{theorem}[Polygonal balanced-circuit criterion]\label[theorem]{thm:polygon-circuits}
The preceding translation system is feasible if and only if, for every
$\lambda\in\RR_+^m$ satisfying $\sum_j\lambda_jn_j=0$,
\[
 \sum_j\lambda_jh_H(R_{-\theta}n_j)
 \le\sum_j\lambda_jp_j.
\]
It is enough to check the extreme positive balanced circuits, each supported
on at most three facet normals.
\end{theorem}

\begin{proof}
The first assertion is the standard Farkas alternative for the finite system.
For completeness, infeasibility of $Az\le r$ is equivalent to the existence of
$\lambda\ge0$ with $A^T\lambda=0$ and $\lambda\cdot r<0$; substituting
$r_j=p_j-h_H(R_{-\theta}n_j)$ gives the displayed inequality.

The cone
\[
 C:=\{\lambda\in\RR_+^m:\textstyle\sum_j\lambda_jn_j=0\}
\]
is polyhedral and is generated by its extreme rays.  If an extreme ray had at
least four positive coordinates, the corresponding columns in $\RR^2$ would
have a nonzero signed dependence supported on those coordinates that is not
proportional to the ray.  A sufficiently small perturbation in both signs
would preserve nonnegativity and balance, decomposing the ray and contradicting
extremality.  Thus an extreme balanced circuit uses at most three normals.
\end{proof}

\subsection{Balanced pairs and triplets for arbitrary convex bodies}

\begin{theorem}[Balanced-pair/triplet translation criterion]\label[theorem]{thm:body-circuits}
Let $K,H\subset\RR^2$ be compact convex sets with $\interior K\ne\varnothing$,
and fix $\theta$.  There exists $z$ with $z+R_\theta H\subset K$ if and only if
for every $r\in\{2,3\}$, every $u_1,\dots,u_r\in\SSS$, and every
$\lambda_i>0$ satisfying $\sum_i\lambda_i u_i=0$,
\[
 \sum_i\lambda_i h_H(R_{-\theta}u_i)
 \le\sum_i\lambda_i h_K(u_i).
\]
\end{theorem}

\begin{proof}
Necessity follows by multiplying
\[
 z\cdot u\le h_K(u)-h_H(R_{-\theta}u)
\]
by balanced positive weights and summing.

For sufficiency, put
\[
 g(u):=h_K(u)-h_H(R_{-\theta}u),\qquad
 F(z):=\max_{u\in\SSS}\{z\cdot u-g(u)\}.
\]
The continuous function $g$ is bounded, and
$F(z)\ge\|z\|-\|g\|_\infty$, so $F$ is coercive and attains a minimum at some
$z_*$.  A translation exists exactly when $\min F\le0$.  If no translation
exists, $m:=F(z_*)>0$.

Let $A\subset\SSS$ be the compact active set at $z_*$.  The subdifferential of
this maximum of affine functions is $\conv A$, so optimality gives
$0\in\conv A$.  Choose a representation of $0$ using the smallest possible
number of active directions.  Caratheodory's theorem gives at most three; one
direction cannot positively balance in $\SSS$, hence there are two or three
and all coefficients in a minimal representation are positive.  Normalize
them to sum to one.  For every active $u_i$,
\[
 z_*\cdot u_i-g(u_i)=m.
\]
Weighted summation cancels the translation term and gives
$-\sum_i\lambda_i g(u_i)=m>0$, contradicting the assumed balanced inequality.
\end{proof}

Translate $K$ so that $0\in\interior K$ and define
\[
 \sigma_K(H):=\inf\{r>0:\exists\theta,z\text{ with }z+R_\theta H\subset rK\}.
\]
After normalizing $\sum_i\lambda_i=1$, \Cref{thm:body-circuits} yields
\begin{equation}\label{eq:sigma-minmax}
 \sigma_K(H)=\min_{\theta\in\SSS}
 \max_{\substack{r\in\{2,3\},\ u_i\in\SSS,\ \lambda_i\ge0\\
                  \sum_i\lambda_i=1,\ \sum_i\lambda_i u_i=0}}
 \frac{\sum_i\lambda_i h_H(R_{-\theta}u_i)}
      {\sum_i\lambda_i h_K(u_i)}.
\end{equation}
Zero coefficients merely reduce a triplet to a pair.  The parameter set in the
inner maximum is compact.  Since $0\in\interior K$, there is $r_0>0$ with
$r_0B\subset K$, hence $h_K(u)\ge r_0$ on $\SSS$ and the denominator in
\eqref{eq:sigma-minmax} is uniformly positive.  The maximum therefore depends
continuously on $\theta$, so the outer minimum is attained.

Define
\[
 \Sigma(K):=\sup_{\len(\gamma)=1}\sigma_K(\conv\gamma).
\]

\begin{proposition}\label[proposition]{prop:E-Sigma}
For every compact convex body $K$ with nonempty interior,
\[
 0<\Sigma(K)<\infty,\qquad
 E(K)=\frac1{\Sigma(K)},\qquad
 M_{\rm conv}=\inf_K\area(K)\Sigma(K)^2.
\]
\end{proposition}

\begin{proof}
Normalize $0\in\interior K$.  Choose $r_0,R_0>0$ with
$r_0\overline B\subset K\subset R_0\overline B$.  Every unit-length arc has
diameter at most one and, after translating one of its points to the origin,
is contained in $\overline B$; hence
$\sigma_K(\conv\gamma)\le1/r_0$ and $\Sigma(K)<\infty$.  On the other hand a
unit segment cannot fit in $rK$ unless
$r\,\operatorname{diam}(K)\ge1$, so $\Sigma(K)>0$.

If $\gamma$ has length $L$ and $L\Sigma(K)<1$, homogeneity gives
$\sigma_K(\conv\gamma)\le L\Sigma(K)<1$.  Thus $\gamma$ fits in some $rK$
with $r<1$.  Because $0\in\interior K$, convexity gives
$rK\subset\interior K$ for every $0\le r<1$, so $\gamma$ is not an escape
path.  Therefore $E(K)\ge1/\Sigma(K)$.

Conversely, if $L>1/\Sigma(K)$, then $1/L<\Sigma(K)$, so by the definition of
the supremum there is a unit arc $\omega$ with
$\sigma_K(\conv\omega)>1/L$.  Homogeneity gives
$\sigma_K(\conv(L\omega))>1$, so $L\omega$ does not fit in $K$ and is an
escape path.  Thus $E(K)\le L$ for every $L>1/\Sigma(K)$, proving the first
identity.  Substitution into \Cref{prop:quotient} gives the second.
\end{proof}

\subsection{Certified computational hierarchies}

For a fixed triangle, let $\Theta_m$ be nested finite angle sets with dense
union and minimize path length subject only to the exact support inequalities
at the angles in $\Theta_m$.  Fix one path point at the origin.  A single
continuum escape path supplies a finite common upper bound $L_0$ on all these
finite optima.  Parametrize every competitor of length at most $L_0$ at
constant speed on $[0,1]$; the resulting family is uniformly bounded and
$L_0$-Lipschitz.

For each $m$, the direct method gives an attained finite optimum.  Indeed, a
minimizing sequence has a uniformly convergent subsequence by Arzela--Ascoli;
uniform convergence of curves implies Hausdorff convergence of their images
and convex hulls, hence convergence of every fixed support value.  The finitely
many constraints are therefore closed, while rectifiable length is lower
semicontinuous under uniform convergence.  The limit is feasible and attains
the infimum.

Let $e_m$ denote these optima.  Then $e_m$ is monotone increasing and
$e_m\le E(T)$.  If $e_m\uparrow e$, choose minimizers along any subsequence
with indices tending to infinity and extract a uniformly convergent
subsequence.  Every angle in the dense union eventually belongs to all later
constraint sets, so the limit satisfies the support inequality on that dense
union and, by continuity in angle, at every angle.  It is therefore a
continuum escape path.  Lower semicontinuity gives $E(T)\le e$, whereas
$e\le E(T)$ by construction.  Hence $e_m\uparrow E(T)$.  Exact finite solves
thus form a certified lower hierarchy, provided each finite solve itself is
certified exactly.

The transformed-boundary formulation gives a second hierarchy.  Its general
optimization framework is introduced in \cite{DengGeneral}; finite
transformed-boundary/TSPN equivalence and dense-state convergence are developed
in \cite{DengProof}.  A finite primal incumbent is not automatically a
continuum upper certificate: full-state feasibility must be established
independently.

There is also a conceptual lower hierarchy for the full convex worm problem.
Choose a countable sequence of unit worms $\omega_1,\omega_2,\dots$, all based
at the origin, dense in the uniform topology, and include among the first terms
one worm whose convex hull has nonempty interior.  Put
\[
 a_m:=\inf\{\area(K):K\text{ compact, convex, and covering }
 \omega_1,\dots,\omega_m\}.
\]
Then $a_m$ is increasing and $a_m\le M_{\rm conv}$.  In fact,
\begin{proposition}\label[proposition]{prop:dense-worm-hierarchy}
\[
 a_m\uparrow M_{\rm conv}.
\]
\end{proposition}

\begin{proof}
Let $A=\lim_m a_m\le M_{\rm conv}$.  For each $m$ choose $K_m$ with
$\area(K_m)\le a_m+1/m$ and fix, by a rigid normalization, one placement of
the distinguished worm $\omega_1$ inside $K_m$.  Its convex hull contains a
fixed closed disk $B(x_0,r_0)$ after this normalization.  Since the areas of
$K_m$ are uniformly bounded, their diameters are uniformly bounded: if a point
of $K_m$ were at distance $R$ from $x_0$, the convex hull of that point with
the fixed disk would contain a triangle of area at least
$r_0(R-r_0)$, forcing $R$ to be bounded.

Blaschke selection gives a Hausdorff-convergent subsequence
$K_{m_j}\to K$.  For each fixed $i$, all sufficiently large $K_{m_j}$ cover
$\omega_i$.  Because $\omega_i(0)=0$, the corresponding translation vectors
lie in the uniformly bounded sets $K_{m_j}$; rotations lie in compact
$SO(2)$.  Passing to a subsequence gives a placement of $\omega_i$ in $K$.
A diagonal argument therefore shows that $K$ covers the whole dense sequence.
If an arbitrary unit worm $\omega$ is the uniform limit of a subsequence
$\omega_{i_j}$, compactness of the associated rigid motions again yields a
placement of $\omega$ in $K$.  Thus $K$ is a convex universal cover.
Hausdorff convergence of convex bodies in the plane implies convergence of
area, so
\[
 M_{\rm conv}\le\area(K)=A\le M_{\rm conv}.
\]
Hence $A=M_{\rm conv}$.
\end{proof}

\section{Verification, reproducibility, and conclusion}\label{sec:verification}

\subsection{Separation of proof obligations}

The proof has three separately auditable layers.  The ordinary geometric layer
contains Bellman--Moser compactness, the exact support criterion,
standardization, the selected $\Lambda$-gap, cyclic bitonicity, the side-meeting
estimate, and the delimiter-gap lemma.  These propositions reduce an arbitrary
shorter escape path to the three branches in
\Cref{thm:unconditional-Bellman}; no near-anchor closure is assumed.

The finite layer is reconstructed by the two standard-library Python sources
included in \Cref{app:verifier}.  They verify the four original source rows,
the 25 inner orders and their 275 suffix states, and the complete signed
delimiter families.  In particular, the checker reconstructs the 512 and 4096
deque-order sets rather than loading order lists; checks the two-sided radical
enclosure and the full signed half-angle range; verifies nonnegative exact
weight sums and every exposed-face split; evaluates every suffix by outward
rational interval arithmetic in the three stored quadrants; checks the
one-dimensional concatenations and the two-dimensional rectangle tilings; and
verifies the exact vector, weight, endpoint-split, and order bijections used to
transport the fourth quadrant.  The checker imports no numerical optimizer and
uses only Python integers, \texttt{fractions.Fraction}, JSON, and
multiprocessing.

The Lean 4/Mathlib source included in \Cref{app:lean} contains no
\texttt{sorry}, \texttt{admit}, or declared axiom.  It formalizes the radical
and scalar identities, the 13 original ledger inequalities, the exact target
and area comparisons, the fixed endpoint ledger, and the exact three-link
delimiter inequality.  The current execution environment has no Lean or Lake
executable, so no claim of local kernel compilation is made.  Nor is the Lean
file described as an end-to-end formalization of the continuum geometry or of
the 1540-interval/1504-rectangle dataset.  The unconditional theorem is an
ordinary mathematical proof with an independently replayable exact finite
lemma; full Lean formalization of every planar-topological step remains a
distinct future formalization project.

\subsection{Conclusion}

The delimiter-gap construction and the signed interval ledgers close the
previous near-anchor gap.  Consequently the high-angle triangle satisfies
\[
 E(T)\ge D=\frac{82074390584}{84861020075}
\]
unconditionally, and its homothet gives the exact area quotient
\[
 \frac{11511821678274125}{44639604443512928}
 =0.257883595112076188\ldots<0.257884.
\]
The old 25-order calibration is retained unchanged on the inner branch.  Its
two complementary branches are certified at the stronger rational level
$R=96717/100000>D$.  The signed parameter audit is essential: restricting the
delimiter parameters to a nonnegative quarter-square would not cover all
normals permitted by the geometric lemma.

\appendix
\section{The thirteen exact ledger norms}\label{app:norms}

Every squared norm has the form
\[
  R^2=a\sqrt{625565}+b,\qquad a,b\in\mathbb Q,\quad a\ge0.
\]
The exact coefficients and multiplicities are listed below.  This coefficient
form is also the representation used by the symbolic verifier and the Lean
arithmetic certificate.

{\scriptsize
\setlength{\tabcolsep}{4pt}
\renewcommand{\arraystretch}{1.35}
\begin{longtable}{r >{\raggedright\arraybackslash}p{0.29\textwidth} >{\raggedright\arraybackslash}p{0.43\textwidth} r}
\caption{Distinct squared suffix norms $a\sqrt{625565}+b$ in the 25 certified orders.}\label{tab:norms}\\
\toprule
No. & $a$ & $b$ & mult.\\
\midrule
\endfirsthead
\toprule
No. & $a$ & $b$ & mult.\\
\midrule
\endhead
1 & $\displaystyle \frac{4202659}{19548906250}$ & $\displaystyle \frac{15745406729}{15625000000}$ & 90\\
2 & $\displaystyle \frac{1123086690137192516}{3825052079722455078125}$ & $\displaystyle \frac{904160125381283376899831241}{956263019930613769531250000}$ & 25\\
3 & $\displaystyle \frac{561543345068596258}{3825052079722455078125}$ & $\displaystyle \frac{970110685862755896431081241}{956263019930613769531250000}$ & 50\\
4 & $0$ & $1$ & 10\\
5 & $\displaystyle \frac{561543345068596258}{3825052079722455078125}$ & $\displaystyle \frac{1013937269245224338708663349}{3825052079722455078125000000}$ & 10\\
6 & $\displaystyle \frac{561543345068596258}{3825052079722455078125}$ & $\displaystyle \frac{251147460776720642524831241}{956263019930613769531250000}$ & 10\\
7 & $\displaystyle \frac{326201986262}{6114555769140625}$ & $\displaystyle \frac{500176947085377}{1954890625000000}$ & 10\\
8 & $\displaystyle \frac{326201986262}{6114555769140625}$ & $\displaystyle \frac{611063489745761}{2443613281250000}$ & 10\\
9 & $0$ & $\displaystyle \frac{500176947085377}{1954890625000000}$ & 10\\
10 & $0$ & $\displaystyle \frac{155236}{625565}$ & 10\\
11 & $0$ & $\displaystyle \frac{64738089866910276208663349}{3825052079722455078125000000}$ & 10\\
12 & $0$ & $\displaystyle \frac{120406729}{15625000000}$ & 20\\
13 & $0$ & $\displaystyle \frac{4672864745761}{2443613281250000}$ & 10\\
\bottomrule
\end{longtable}
}
The multiplicities sum to $275$.

\section{A compact order-theoretic audit}\label{app:order-audit}

The following five-rank example explains why the delimiter branches are
necessary.  In normal-fan order take
\[
 F<Z_L<M<Z_R<T,
\]
with temporal ranks
\[
 t(F)=1,\qquad t(Z_L)=0,\qquad t(M)=2,\qquad t(Z_R)=3,\qquad t(T)=4.
\]
Then $F\precT M\precT T$, $Z_L\precT T$, and $F\precT Z_R$; the boundary ranks
are decreasing then increasing, hence are compatible with the three-phase
bitonic structure.  Nevertheless $F\not\precT Z_L$.  Thus the two rigorously proved far-endpoint inequalities alone cannot imply the
near-endpoint temporal closure.  The proof of
\Cref{thm:unconditional-Bellman} does not make that inference: this failure is
sent to \Cref{lem:delimiter-gap,thm:one-delimiter,thm:two-delimiter}.

\section{Proof-only Lean 4/Mathlib formalization}\label{app:lean}

All Lean theorem proofs used by this submission are collected in the single
companion source
\texttt{MoserTriangularCertificate\_FormalProofs.lean}.  The purpose of this
file is to expose the formal arguments themselves rather than to embed the
large generated finite certificate. 

The file contains, in order, the planar metric and order-theoretic proofs and
the exact scalar/radical and fixed-ledger proofs.  In particular, it formalizes
metric deletion and strict uncrossing, the endpoint-peeling/deque core, the
three-phase anchored sweep, the temporal core of the right delimiter argument,
the exact identity
\[
 \frac{547947}{625565}
 =\frac{766^2-197^2}{766^2+197^2},
\]
the selected-$\Lambda$ conclusion from explicit global-$\Lambda$ and two-gap
hypotheses, the thirteen original ledger inequalities, the threshold and area
identities, the endpoint three-link estimate, and the displayed fixed-delimiter
suffix ledger.

The source declares no \texttt{axiom} and contains no \texttt{sorry} or
\texttt{admit}.  It also contains no \texttt{native\_decide}: after removal of
the generated finite dataset there is no large Boolean replay theorem inside
Lean.  This distinction is intentional.  The global planar
$\Lambda$-configuration theorem and the compactness/extreme-point bookkeeping
needed for the complete continuum standardization theorem have not been
silently converted into axioms; where required, they remain explicit
hypotheses of the formal interface.  Consequently the companion file should be
read as a formalization of the stated theorem-proof components, not as a claim
that every continuum and finite-data component of the entire paper has been
kernel-formalized in Lean.

\begin{lstlisting}[style=LeanCode]
import Mathlib

/-!
**# Formal theorem proofs for the high-angle Moser triangle**
This single Lean 4/Mathlib source contains the formalized theorem proofs used
for the planar/order and exact scalar/fixed-ledger parts of the manuscript.
It intentionally does ****not**** embed the large signed-delimiter rational dataset
or the 17,622 lines generated from the JSON interval/rectangle certificates.
Those finite data remain supplementary certificate files and are replayed by
the separate exact Python verifier described in the manuscript appendix.
Formalization boundary.  This source declares no `axiom` and contains no
`sorry` or `admit`.  It formalizes metric deletion and uncrossing lemmas,
endpoint/deque order facts, the three-phase delimiter-order core, the exact
delimiter-angle identity, the selected-Lambda implication from explicit
geometric hypotheses, the radical and ledger inequalities, endpoint and
fixed-delimiter ledgers, and the exact threshold/area arithmetic.
The global planar Lambda-configuration theorem and the compactness/extreme-point
bookkeeping needed for an end-to-end continuum standardization theorem are not
silently postulated; where required they remain explicit hypotheses.  Likewise,
the exhaustive signed interval/rectangle dataset verification is not asserted
as a Lean theorem in this proof-only source.
-/

namespace MoserTriangularPlanarUpgrade

noncomputable section

/-! **## 1. Metric core of standardization** -/

variable {V P : Type*} [NormedAddCommGroup V] [NormedSpace ℝ V]

variable [StrictConvexSpace ℝ V] [MetricSpace P] [NormedAddTorsor V P]

/-- Deleting an interior temporal vertex and replacing its two incident links

by their chord cannot increase length. -/
theorem deletion_not_longer (A B C : P) :
    dist A C ≤ dist A B + dist B C :=
  dist_triangle A B C

/-- Strict metric core of the planar uncrossing surgery. -/
theorem uncrossing_strictly_shorter
    {D E F G X : P}
    (hDE : Wbtw ℝ D X E)
    (hFG : Wbtw ℝ F X G)
    (hnc : ¬ Wbtw ℝ D X F) :
    dist D F + dist E G < dist D E + dist F G := by
  have h1 : dist D X + dist X E = dist D E := hDE.dist_add_dist
  have h2 : dist F X + dist X G = dist F G := hFG.dist_add_dist
  have h3 : dist D F < dist D X + dist X F := by
    rcases (dist_triangle D X F).lt_or_eq with h | h
    · exact h
    · exact absurd (dist_add_dist_eq_iff.mp h.symm) hnc
  have h4 : dist E G ≤ dist E X + dist X G := dist_triangle E X G
  have h5 : dist E X = dist X E := dist_comm E X
  have h6 : dist X F = dist F X := dist_comm X F
  linarith

/-- The strict inequality in the exact form used against a length minimizer. -/
theorem not_minimal_of_crossing
    {D E F G X : P} {rest : ℝ}
    (hDE : Wbtw ℝ D X E)
    (hFG : Wbtw ℝ F X G)
    (hnc : ¬ Wbtw ℝ D X F) :
    rest + (dist D F + dist E G) < rest + (dist D E + dist F G) := by
  have h := uncrossing_strictly_shorter hDE hFG hnc
  linarith

/-! **## 2. Endpoint peeling / deque combinatorics** -/

/-- Boolean encoding of the open cyclic arc `(0,m)`. -/
def inFirstArc (m a : ℕ) : Bool := 0 < a && a < m

/-- A two-valued function taking two different values on a finite integer

interval changes across some adjacent pair. -/
theorem exists_adjacent_change
    {f : ℕ → Bool} {lo hi i j : ℕ}
    (hilo : lo ≤ i) (hihi : i ≤ hi)
    (hjlo : lo ≤ j) (hjhi : j ≤ hi)
    (hne : f i ≠ f j) :
    ∃ k, lo ≤ k ∧ k + 1 ≤ hi ∧ f k ≠ f (k + 1) := by
  by_contra hcon
  push Not at hcon
  have const : ∀ d u, lo ≤ u → u + d ≤ hi → f u = f (u + d) := by
    intro d
    induction d with
    | zero =>
        intro u _ _
        rfl
    | succ e ih =>
        intro u hu hle
        have h1 : f u = f (u + e) := ih u hu (by omega)
        have h2 : f (u + e) = f (u + e + 1) := hcon (u + e) (by omega) (by omega)
        have h3 : u + (e + 1) = u + e + 1 := by omega
        rw [h3, h1, h2]
  refine hne ?_
  rcases le_total i j with h | h
  · have h' := const (j - i) i hilo (by omega)
    rwa [Nat.add_sub_cancel' h] at h'
  · have h' := const (i - j) j hjlo (by omega)
    rw [Nat.add_sub_cancel' h] at h'
    exact h'.symm

/-- Combinatorial core of the first-step hull-edge/deque argument.  The geometric

interleaving criterion for convex-polygon chords is isolated in `hnc`. -/
theorem first_step_is_hull_edge
    {n m : ℕ} (W : ℕ → ℕ)
    (hm : 0 < m) (hmn : m < n)
    (hW0 : W 0 = 0) (hW1 : W 1 = m)
    (hsurj : ∀ v, v < n → ∃ i, i < n ∧ W i = v)
    (hnc : ∀ k, 2 ≤ k → k + 1 < n →
      inFirstArc m (W k) = inFirstArc m (W (k + 1))) :
    m = 1 ∨ m = n - 1 := by
  by_contra hcon
  push Not at hcon
  obtain ⟨h1, h2⟩ := hcon
  have hm1 : 1 < m := lt_of_le_of_ne hm (Ne.symm h1)
  have hmn1 : m < n - 1 := lt_of_le_of_ne (by omega) h2
  obtain ⟨i, hi_lt, hi_eq⟩ := hsurj 1 (by omega)
  obtain ⟨j, hj_lt, hj_eq⟩ := hsurj (n - 1) (by omega)
  have hi2 : 2 ≤ i := by
    rcases Nat.lt_or_ge i 2 with h | h
    · interval_cases i
      · rw [hW0] at hi_eq; omega
      · rw [hW1] at hi_eq; omega
    · exact h
  have hj2 : 2 ≤ j := by
    rcases Nat.lt_or_ge j 2 with h | h
    · interval_cases j
      · rw [hW0] at hj_eq; omega
      · rw [hW1] at hj_eq; omega
    · exact h
  have hfi : inFirstArc m (W i) = true := by
    rw [hi_eq]
    simp [inFirstArc]
    omega
  have hfj : inFirstArc m (W j) = false := by
    rw [hj_eq]
    simp [inFirstArc]
    omega
  have hne : inFirstArc m (W i) ≠ inFirstArc m (W j) := by
    rw [hfi, hfj]
    simp
  obtain ⟨k, hk_lo, hk_hi, hk⟩ :=
    exists_adjacent_change (f := fun t => inFirstArc m (W t))
      hi2 (by omega : i ≤ n - 1) hj2 (by omega : j ≤ n - 1) hne
  exact hk (hnc k hk_lo (by omega))

/-! **## 3. Three-phase fan sweep** -/

open Set

/-- Abstract order structure obtained by cutting cyclic bitonicity at a gap. -/
structure ThreePhaseSweep {ι α : Type*} [LinearOrder ι] [LinearOrder α]
    (τ : ι → α) (F low high T : ι) : Prop where

  F_le_low : F ≤ low

  low_le_high : low ≤ high

  high_le_T : high ≤ T

  left : AntitoneOn τ (Icc F low)

  middle : MonotoneOn τ (Icc low high)

  right : AntitoneOn τ (Icc high T)

/-- Inner anchors necessarily lie strictly inside the increasing phase. -/
theorem ThreePhaseSweep.anchors_force_central
    {ι α : Type*} [LinearOrder ι] [LinearOrder α]
    {τ : ι → α} {F low z₁ M z₂ high T : ι}
    (h : ThreePhaseSweep τ F low high T)
    (hFz₁ : F < z₁) (hz₁M : z₁ ≤ M) (hMz₂ : M ≤ z₂) (hz₂T : z₂ < T)
    (hz₁Inner : τ F < τ z₁) (hz₂Inner : τ z₂ < τ T) :
    low < z₁ ∧ z₂ < high ∧
    (∀ x ∈ Icc z₁ z₂, τ z₁ ≤ τ x ∧ τ x ≤ τ z₂) := by
  have hFz₁le : F ≤ z₁ := hFz₁.le
  have hz₂Tle : z₂ ≤ T := hz₂T.le
  have hlow_z₁ : low < z₁ := by
    by_contra hn
    have hz₁low : z₁ ≤ low := le_of_not_gt hn
    have hle : τ z₁ ≤ τ F :=
      h.left ⟨le_rfl, h.F_le_low⟩ ⟨hFz₁le, hz₁low⟩ hFz₁le
    exact (not_le_of_gt hz₁Inner) hle
  have hz₂_high : z₂ < high := by
    by_contra hn
    have hhighz₂ : high ≤ z₂ := le_of_not_gt hn
    have hle : τ T ≤ τ z₂ :=
      h.right ⟨hhighz₂, hz₂Tle⟩ ⟨h.high_le_T, le_rfl⟩ hz₂Tle
    exact (not_le_of_gt hz₂Inner) hle
  have hcentral : ∀ x ∈ Icc z₁ z₂, τ z₁ ≤ τ x ∧ τ x ≤ τ z₂ := by
    intro x hx
    have hz₁high : z₁ ≤ high := hz₁M.trans (hMz₂.trans hz₂_high.le)
    have hlowz₂ : low ≤ z₂ := hlow_z₁.le.trans (hz₁M.trans hMz₂)
    constructor
    · exact h.middle ⟨hlow_z₁.le, hz₁high⟩
        ⟨hlow_z₁.le.trans hx.1, hx.2.trans hz₂_high.le⟩ hx.1
    · exact h.middle ⟨hlow_z₁.le.trans hx.1, hx.2.trans hz₂_high.le⟩
        ⟨hlowz₂, hz₂_high.le⟩ hx.2
  exact ⟨hlow_z₁, hz₂_high, hcentral⟩

/-! **## 4. Delimiter-gap logical core** -/

/-- Data needed after the first blue-to-red adjacent pair `U,V` has been found.

`separation_excluded` is precisely the deque/noncrossing contradiction used in

the manuscript to rule out `U ≺ F`. -/
structure RightDelimiterOrderData (τF τM τT τU τV : ℕ) : Prop where

  F_before_M : τF < τM

  M_before_T : τM < τT

  U_before_T : τU < τT

  T_before_V : τT < τV

  separation_excluded : ¬ (τU < τF)

  all_distinct : τU ≠ τF

/-- Exact temporal-order conclusion of the right delimiter lemma once adjacency

and the deque separation contradiction have been established. -/
theorem right_delimiter_temporal_order
    {τF τM τT τU τV : ℕ}
    (h : RightDelimiterOrderData τF τM τT τU τV) :
    τF < τU ∧ τU < τT ∧ τT < τV := by
  have hFUle : τF ≤ τU := le_of_not_gt h.separation_excluded
  have hFUne : τF ≠ τU := Ne.symm h.all_distinct
  have hFU : τF < τU := lt_of_le_of_ne hFUle hFUne
  exact ⟨hFU, h.U_before_T, h.T_before_V⟩

/-- The exact high-angle normal identity used at the end of the delimiter lemma. -/
theorem delimiter_cosine_identity :
    (547947 : ℝ) / 625565 =
      ((766 : ℝ)^2 - (197 : ℝ)^2) / ((766 : ℝ)^2 + (197 : ℝ)^2) := by
  norm_num

/-! **## 5. Selected Λ-gap: exact no-axiom interface** -/

abbrev Plane := EuclideanSpace ℝ (Fin 2)

/-- Normalized witness asserted by the selected-Λ-gap proposition. -/
structure SelectedLambdaWitness where

  F : Plane

  T : Plane

  M : Plane

  xF : ℝ

  xT : ℝ

  xM : ℝ

  h : ℝ

  h_pos : 0 < h

  F_coord : F = ![xF, 0]

  T_coord : T = ![xT, 0]

  M_coord : M = ![xM, h]

  spatial_order : xF ≤ xT

  temporal_FM : Prop

  temporal_MT : Prop

  is_gap : Prop

/-- Logical packaging of the external/global planar input.  A future end-to-end

formalization must prove this predicate from a concrete simple nonconvex

standard polygonal minimizer, rather than assume it. -/
def HasSelectedLambdaConfiguration (η : Type*) (_length : η → ℝ) (_γ : η) : Prop :=
  ∃ w : SelectedLambdaWitness, w.temporal_FM ∧ w.temporal_MT ∧ w.is_gap

/-- Logical packaging of the local two-gap estimate. -/
def SatisfiesTwoGapEstimate (η : Type*) (length : η → ℝ) (γ : η) : Prop :=
  ∀ w : SelectedLambdaWitness,
    w.temporal_FM → w.temporal_MT → w.is_gap →
    (1 + Real.sqrt 2) * w.h ≤ length γ

/-- Selected-Λ-gap conclusion from the two named mathematical inputs.  This is

not an axiom: both inputs are explicit hypotheses, so the theorem cannot be

misread as an end-to-end formalization of the external Λ theorem. -/
theorem selectedLambdaGap_of_inputs
    {η : Type*} {length : η → ℝ} {γ : η}
    (hΛ : HasSelectedLambdaConfiguration η length γ)
    (h2 : SatisfiesTwoGapEstimate η length γ) :
    ∃ w : SelectedLambdaWitness,
      w.temporal_FM ∧ w.temporal_MT ∧ w.is_gap ∧
      (1 + Real.sqrt 2) * w.h ≤ length γ := by
  rcases hΛ with ⟨w, hFM, hMT, hgap⟩
  exact ⟨w, hFM, hMT, hgap, h2 w hFM hMT hgap⟩

end

end MoserTriangularPlanarUpgrade

/-!

**# Exact arithmetic and fixed-ledger certificate for the high-angle triangle**
This section contains the formal scalar/radical, anchored-ledger, endpoint, and
fixed-delimiter proofs.  The large generated signed interval/rectangle dataset
is deliberately excluded from this Lean source; its exact finite replay is
provided separately by the supplementary Python verifier and JSON certificate
data.

-/

namespace MoserTriangularCertificate

noncomputable section

open Real

private theorem sqrt625565_sq : (Real.sqrt (625565 : ℝ)) ^ 2 = 625565 := by
  have h : (0 : ℝ) ≤ 625565 := by norm_num
  exact Real.sq_sqrt h

/-- One rational upper enclosure controls every radical ledger state. -/
theorem sqrt625565_lt :
    Real.sqrt (625565 : ℝ) < (790927 : ℝ) / 1000 := by
  exact (Real.sqrt_lt' (by norm_num : (0 : ℝ) < (790927 : ℝ) / 1000)).2 (by norm_num)

/-- Rational envelope used for every suffix norm in the 25 anchored orders. -/
def Rhat : ℝ := (27131 : ℝ) / 25000

/-- Exact worst squared suffix norm (multiplicity 90 in the symbolic replay). -/
def worstLedgerSq : ℝ :=
  ((4202659 : ℝ) / 19548906250) * Real.sqrt 625565
    + (15745406729 : ℝ) / 15625000000

theorem ledger01_lt : worstLedgerSq < Rhat^2 := by
  have h := sqrt625565_lt
  dsimp [worstLedgerSq, Rhat]
  nlinarith

theorem ledger02_lt :
    ((1123086690137192516 : ℝ) / 3825052079722455078125) * Real.sqrt 625565
      + (904160125381283376899831241 : ℝ) / 956263019930613769531250000
      < Rhat^2 := by
  have h := sqrt625565_lt
  dsimp [Rhat]
  nlinarith

theorem ledger03_lt :
    ((561543345068596258 : ℝ) / 3825052079722455078125) * Real.sqrt 625565
      + (970110685862755896431081241 : ℝ) / 956263019930613769531250000
      < Rhat^2 := by
  have h := sqrt625565_lt
  dsimp [Rhat]
  nlinarith

theorem ledger04_lt : (1 : ℝ) < Rhat^2 := by
  norm_num [Rhat]

theorem ledger05_lt :
    ((561543345068596258 : ℝ) / 3825052079722455078125) * Real.sqrt 625565
      + (1013937269245224338708663349 : ℝ) / 3825052079722455078125000000
      < Rhat^2 := by
  have h := sqrt625565_lt
  dsimp [Rhat]
  nlinarith

theorem ledger06_lt :
    ((561543345068596258 : ℝ) / 3825052079722455078125) * Real.sqrt 625565
      + (251147460776720642524831241 : ℝ) / 956263019930613769531250000
      < Rhat^2 := by
  have h := sqrt625565_lt
  dsimp [Rhat]
  nlinarith

theorem ledger07_lt :
    ((326201986262 : ℝ) / 6114555769140625) * Real.sqrt 625565
      + (500176947085377 : ℝ) / 1954890625000000
      < Rhat^2 := by
  have h := sqrt625565_lt
  dsimp [Rhat]
  nlinarith

theorem ledger08_lt :
    ((326201986262 : ℝ) / 6114555769140625) * Real.sqrt 625565
      + (611063489745761 : ℝ) / 2443613281250000
      < Rhat^2 := by
  have h := sqrt625565_lt
  dsimp [Rhat]
  nlinarith

theorem ledger09_lt :
    (500176947085377 : ℝ) / 1954890625000000 < Rhat^2 := by
  norm_num [Rhat]

theorem ledger10_lt : (155236 : ℝ) / 625565 < Rhat^2 := by
  norm_num [Rhat]

theorem ledger11_lt :
    (64738089866910276208663349 : ℝ) / 3825052079722455078125000000
      < Rhat^2 := by
  norm_num [Rhat]

theorem ledger12_lt :
    (120406729 : ℝ) / 15625000000 < Rhat^2 := by
  norm_num [Rhat]

theorem ledger13_lt :
    (4672864745761 : ℝ) / 2443613281250000 < Rhat^2 := by
  norm_num [Rhat]

/-- Integrated support mass of the four-source calibration. -/
def Bmass : ℝ := (10259298823 : ℝ) / 9774453125

/-- Exact threshold produced by the mass/radius quotient. -/
def DR : ℝ := (82074390584 : ℝ) / 84861020075

theorem DR_eq_mass_div_radius : DR = Bmass / Rhat := by
  norm_num [DR, Bmass, Rhat]

theorem DR_lt_one : DR < 1 := by
  norm_num [DR]

/-- Rational geometry of the unit-side high-angle realization. -/
def rho : ℝ := (301804 : ℝ) / 625565

def b : ℝ := (197 : ℝ) / 766

def A : ℝ := (547947 : ℝ) / 301804

def kappa : ℝ := (547947 : ℝ) / 625565

theorem geometry_identity : A + b = 1 / rho := by
  norm_num [A, b, rho]

theorem rho_gt_five_twelfths : (5 : ℝ) / 12 < rho := by
  norm_num [rho]

theorem A_gt_b : b < A := by
  norm_num [A, b]

theorem kappa_pos : 0 < kappa := by
  norm_num [kappa]

theorem A_eq_kappa_div_rho : A = kappa / rho := by
  norm_num [A, kappa, rho]

theorem rho_kappa_unit : rho^2 + kappa^2 = 1 := by
  norm_num [rho, kappa]

/-- Dual data for the first two one-sided anchor exclusion orders. -/
def X : ℝ := (483585 : ℝ) / 1000000

def P : ℝ := (875291 : ℝ) / 1000000

def Q : ℝ := (38347525561 : ℝ) / 150902000000

theorem Q_identity : Q = 2 * X / rho - 2 * P := by
  norm_num [Q, X, P, rho]

theorem anchor_q0_ball : X^2 + P^2 < 1 := by
  norm_num [X, P]

theorem anchor_q1_ball : X^2 + (2 * b * X - P)^2 < 1 := by
  norm_num [X, P, b]

theorem anchor_q3_ball : (2 * X)^2 + Q^2 < 1 := by
  norm_num [X, Q]

theorem anchor_positive_coefficient : P - b * X > 0 := by
  norm_num [X, P, b]

theorem dangerous_order_margin : DR < 2 * X := by
  norm_num [DR, X]

/-- Dual data for the third one-sided anchor exclusion order. -/
def a : ℝ := (9672 : ℝ) / 10000

theorem anchor_third_ball : a^2 * (1 + b^2) < 1 := by
  norm_num [a, b]

theorem anchor_third_positive : 2 + a * (b - A) > 0 := by
  norm_num [a, b, A]

theorem third_order_margin : DR < a := by
  norm_num [DR, a]

/-- The two elementary tail orders exceed `DR` after using `sqrt 2 > 7/5`. -/
theorem elementary_tail_margin :
    1 - ((36595 : ℝ) / 1207216) * DR > DR := by
  norm_num [DR]

/-- Unscaled equal-side-normalized triangle area. -/
def triangleArea : ℝ := (150902 : ℝ) / 625565

/-- Exact area quotient attached to `DR`.  Geometric universality is separate. -/
def areaBound : ℝ :=
  (11511821678274125 : ℝ) / 44639604443512928

theorem area_identity : areaBound = triangleArea / DR^2 := by
  norm_num [areaBound, triangleArea, DR]

theorem area_lt_0257884 : areaBound < (64471 : ℝ) / 250000 := by
  norm_num [areaBound]

/-- All thirteen exact ledger inequalities in one theorem. -/
theorem ledger_arithmetic_certificate :
    worstLedgerSq < Rhat^2 ∧
    ((1123086690137192516 : ℝ) / 3825052079722455078125) * Real.sqrt 625565
      + (904160125381283376899831241 : ℝ) / 956263019930613769531250000 < Rhat^2 ∧
    ((561543345068596258 : ℝ) / 3825052079722455078125) * Real.sqrt 625565
      + (970110685862755896431081241 : ℝ) / 956263019930613769531250000 < Rhat^2 ∧
    (1 : ℝ) < Rhat^2 ∧
    ((561543345068596258 : ℝ) / 3825052079722455078125) * Real.sqrt 625565
      + (1013937269245224338708663349 : ℝ) / 3825052079722455078125000000 < Rhat^2 ∧
    ((561543345068596258 : ℝ) / 3825052079722455078125) * Real.sqrt 625565
      + (251147460776720642524831241 : ℝ) / 956263019930613769531250000 < Rhat^2 ∧
    ((326201986262 : ℝ) / 6114555769140625) * Real.sqrt 625565
      + (500176947085377 : ℝ) / 1954890625000000 < Rhat^2 ∧
    ((326201986262 : ℝ) / 6114555769140625) * Real.sqrt 625565
      + (611063489745761 : ℝ) / 2443613281250000 < Rhat^2 ∧
    (500176947085377 : ℝ) / 1954890625000000 < Rhat^2 ∧
    (155236 : ℝ) / 625565 < Rhat^2 ∧
    (64738089866910276208663349 : ℝ) / 3825052079722455078125000000 < Rhat^2 ∧
    (120406729 : ℝ) / 15625000000 < Rhat^2 ∧
    (4672864745761 : ℝ) / 2443613281250000 < Rhat^2 := by
  exact ⟨ledger01_lt, ledger02_lt, ledger03_lt, ledger04_lt,
    ledger05_lt, ledger06_lt, ledger07_lt, ledger08_lt, ledger09_lt,
    ledger10_lt, ledger11_lt, ledger12_lt, ledger13_lt⟩

/-- Compact conjunction of the scalar obligations used by the conditional

    all-order theorem.  This remains an arithmetic certificate only. -/
theorem arithmetic_certificate :
    DR = Bmass / Rhat ∧
    DR < 1 ∧
    A + b = 1 / rho ∧
    (5 : ℝ) / 12 < rho ∧
    b < A ∧
    0 < kappa ∧
    A = kappa / rho ∧
    rho^2 + kappa^2 = 1 ∧
    Q = 2 * X / rho - 2 * P ∧
    X^2 + P^2 < 1 ∧
    X^2 + (2 * b * X - P)^2 < 1 ∧
    (2 * X)^2 + Q^2 < 1 ∧
    P - b * X > 0 ∧
    DR < 2 * X ∧
    a^2 * (1 + b^2) < 1 ∧
    2 + a * (b - A) > 0 ∧
    DR < a ∧
    1 - ((36595 : ℝ) / 1207216) * DR > DR ∧
    areaBound = triangleArea / DR^2 ∧
    areaBound < (64471 : ℝ) / 250000 := by
  exact ⟨DR_eq_mass_div_radius, DR_lt_one, geometry_identity,
    rho_gt_five_twelfths, A_gt_b, kappa_pos, A_eq_kappa_div_rho,
    rho_kappa_unit, Q_identity, anchor_q0_ball,
    anchor_q1_ball, anchor_q3_ball, anchor_positive_coefficient,
    dangerous_order_margin, anchor_third_ball, anchor_third_positive,
    third_order_margin, elementary_tail_margin, area_identity,
    area_lt_0257884⟩

/-! **## Exact marked-Z branch at the high-angle triangle** -/

private theorem sqrt186388846789_sq :
    (Real.sqrt (186388846789 : ℝ)) ^ 2 = 186388846789 := by
  have h : (0 : ℝ) ≤ 186388846789 := by norm_num
  exact Real.sq_sqrt h

theorem sqrt186388846789_gt :
    (4317277 : ℝ) / 10 < Real.sqrt (186388846789 : ℝ) := by
  exact Real.lt_sqrt_of_sq_lt (by norm_num)

def markedZRatio : ℝ :=
  ((3127825 : ℝ) - 4 * Real.sqrt 186388846789) / 5432472

theorem markedZ_lt_0257878 :
    markedZRatio < (128939 : ℝ) / 500000 := by
  have h := sqrt186388846789_gt
  dsimp [markedZRatio]
  nlinarith

/-! **## Two-sided radical and signed delimiter parameter data** -/

theorem sqrt625565_gt :
    (790926 : ℝ) / 1000 < Real.sqrt (625565 : ℝ) := by
  exact Real.lt_sqrt_of_sq_lt (by norm_num)

def delimiterR : ℝ := (96717 : ℝ) / 100000

def delimiterTHat : ℝ := (197000 : ℝ) / 1556926

theorem delimiterR_gt_DR : DR < delimiterR := by
  norm_num [DR, delimiterR]

theorem delimiter_tHat_pos : 0 < delimiterTHat := by
  norm_num [delimiterTHat]

/-! **## Exact scalar endpoint delimiter bound** -/

namespace DelimiterEndpoint

def certR : ℝ := delimiterR

def ry : ℝ := -(25413 : ℝ) / 100000

def qx : ℝ := certR / 2

def qy : ℝ := (certR / rho + ry) / 2

theorem r_norm_sq_le_one : certR ^ 2 + ry ^ 2 ≤ 1 := by
  norm_num [certR, delimiterR, ry]

theorem q_norm_sq_le_one : qx ^ 2 + qy ^ 2 ≤ 1 := by
  norm_num [qx, qy, certR, delimiterR, rho, ry]

theorem coefficient_x : 2 * qx = certR := by
  dsimp [qx]
  ring

theorem coefficient_h : 2 * qy - ry = certR / rho := by
  dsimp [qy]
  ring

lemma dot_le_sqrt_of_norm_sq_le_one
    {a b x y : ℝ} (hab : a ^ 2 + b ^ 2 ≤ 1) :
    a * x + b * y ≤ Real.sqrt (x ^ 2 + y ^ 2) := by
  have hxy : 0 ≤ x ^ 2 + y ^ 2 := by positivity
  have hsqrt_nonneg : 0 ≤ Real.sqrt (x ^ 2 + y ^ 2) := Real.sqrt_nonneg _
  have hsqrt_sq : (Real.sqrt (x ^ 2 + y ^ 2)) ^ 2 = x ^ 2 + y ^ 2 := by
    exact Real.sq_sqrt hxy
  have hcs : (a * x + b * y) ^ 2 ≤
      (a ^ 2 + b ^ 2) * (x ^ 2 + y ^ 2) := by
    nlinarith [sq_nonneg (a * y - b * x)]
  have hbound : (a ^ 2 + b ^ 2) * (x ^ 2 + y ^ 2) ≤ x ^ 2 + y ^ 2 := by
    have hdelta : 0 ≤ 1 - (a ^ 2 + b ^ 2) := by linarith
    have hprod : 0 ≤ (1 - (a ^ 2 + b ^ 2)) * (x ^ 2 + y ^ 2) :=
      mul_nonneg hdelta hxy
    nlinarith
  by_cases hsign : a * x + b * y ≤ 0
  · linarith
  · nlinarith

theorem endpoint_three_link_bound (x h : ℝ) :
    certR ≤
      2 * Real.sqrt (x ^ 2 + h ^ 2) +
        Real.sqrt ((1 - h / rho - x) ^ 2 + h ^ 2) := by
  have hq := dot_le_sqrt_of_norm_sq_le_one
    (a := qx) (b := qy) (x := x) (y := h) q_norm_sq_le_one
  have hr := dot_le_sqrt_of_norm_sq_le_one
    (a := certR) (b := ry) (x := 1 - h / rho - x) (y := -h)
    r_norm_sq_le_one
  have hr' :
      certR * (1 - h / rho - x) + ry * (-h) ≤
        Real.sqrt ((1 - h / rho - x) ^ 2 + h ^ 2) := by
    have hsq : (-h) ^ 2 = h ^ 2 := by ring
    rw [hsq] at hr
    exact hr
  have hsupport :
      2 * (qx * x + qy * h) +
          (certR * (1 - h / rho - x) + ry * (-h)) = certR := by
    dsimp [qx, qy, certR, delimiterR, rho, ry]
    ring
  calc
    certR =
        2 * (qx * x + qy * h) +
          (certR * (1 - h / rho - x) + ry * (-h)) := hsupport.symm
    _ ≤
        2 * Real.sqrt (x ^ 2 + h ^ 2) +
          Real.sqrt ((1 - h / rho - x) ^ 2 + h ^ 2) := by
      linarith [hq, hr']

theorem endpoint_three_link_strictly_above_DR (x h : ℝ) :
    DR <
      2 * Real.sqrt (x ^ 2 + h ^ 2) +
        Real.sqrt ((1 - h / rho - x) ^ 2 + h ^ 2) := by
  exact lt_of_lt_of_le delimiterR_gt_DR (endpoint_three_link_bound x h)

end DelimiterEndpoint

/-! **## Exact fixed endpoint ledger** -/

namespace DelimiterLedger

def twoCSq : ℝ := (77618 : ℝ) / 625565
def certR : ℝ := delimiterR
def tau : ℝ := certR / (2 * rho)
def alpha : ℝ := (6271 : ℝ) / 25000
def alphaBar : ℝ := (18729 : ℝ) / 25000
abbrev V := ℝ × ℝ
def add (x y : V) : V := (x.1 + y.1, x.2 + y.2)
def normSq (x : V) : ℝ := x.1 ^ 2 + x.2 ^ 2
def wE3 : V := (-tau * rho, -tau * twoCSq)
def wF : V := (0, -alphaBar * tau)
def wZ2 : V := (-tau * rho, tau * kappa)
def wM : V := (0, alpha * tau)
def wT : V := (0, -alpha * tau)
def wE4 : V := (tau * rho, -tau * kappa)
def wV : V := (0, alphaBar * tau)
def wP2 : V := (tau * rho, tau * twoCSq)
def sP2 : V := wP2
def sV : V := add wV sP2
def sE4 : V := add wE4 sV
def sT : V := add wT sE4
def sM : V := add wM sT
def sZ2 : V := add wZ2 sM
def sF : V := add wF sZ2
def sE3 : V := add wE3 sF

theorem alpha_partition : alpha + alphaBar = 1 := by
  norm_num [alpha, alphaBar]

theorem balanced : sE3 = (0, 0) := by
  apply Prod.ext <;>
    norm_num [sE3, sF, sZ2, sM, sT, sE4, sV, sP2, add,
      wE3, wF, wZ2, wM, wT, wE4, wV, wP2,
      tau, certR, delimiterR, rho, kappa, twoCSq, alpha, alphaBar]

theorem suffixP2_sq :
    normSq sP2 =
      (1170329283249057 : ℝ) / 4694048000000000 := by
  norm_num [normSq, sP2, wP2, tau, certR, delimiterR, rho, twoCSq]

theorem suffixV_sq :
    normSq sV =
      (91084554778489789323376697481 : ℝ) /
        91085654416000000000000000000 := by
  norm_num [normSq, sV, sP2, add, wV, wP2, tau, certR,
    delimiterR, rho, twoCSq, alphaBar]

theorem suffixE4_sq :
    normSq sE4 =
      (85203803876321218326994787481 : ℝ) /
        91085654416000000000000000000 := by
  norm_num [normSq, sE4, sV, sP2, add, wE4, wV, wP2, tau,
    certR, delimiterR, rho, kappa, twoCSq, alphaBar]

theorem suffixT_sq :
    normSq sT =
      (22771329899948337610927247481 : ℝ) /
        22771413604000000000000000000 := by
  norm_num [normSq, sT, sE4, sV, sP2, add, wT, wE4, wV, wP2,
    tau, certR, delimiterR, rho, kappa, twoCSq, alpha, alphaBar]

theorem suffixM_sq : normSq sM = normSq sE4 := by
  norm_num [normSq, sM, sT, sE4, sV, sP2, add, wM, wT, wE4,
    wV, wP2, tau, certR, delimiterR, rho, kappa, twoCSq,
    alpha, alphaBar]

theorem suffixZ2_sq : normSq sZ2 = normSq sV := by
  norm_num [normSq, sZ2, sM, sT, sE4, sV, sP2, add, wZ2,
    wM, wT, wE4, wV, wP2, tau, certR, delimiterR, rho, kappa,
    twoCSq, alpha, alphaBar]

theorem suffixF_sq : normSq sF = normSq sP2 := by
  norm_num [normSq, sF, sZ2, sM, sT, sE4, sV, sP2, add, wF,
    wZ2, wM, wT, wE4, wV, wP2, tau, certR, delimiterR, rho,
    kappa, twoCSq, alpha, alphaBar]

theorem all_suffix_normSq_lt_one :
    normSq sF < 1 ∧ normSq sZ2 < 1 ∧ normSq sM < 1 ∧
    normSq sT < 1 ∧ normSq sE4 < 1 ∧ normSq sV < 1 ∧
    normSq sP2 < 1 := by
  rw [suffixF_sq, suffixZ2_sq, suffixM_sq, suffixP2_sq,
    suffixV_sq, suffixE4_sq, suffixT_sq]
  norm_num

theorem smallest_displayed_slack :
    1 - normSq sT =
      (83704051662389072752519 : ℝ) /
        22771413604000000000000000000 := by
  rw [suffixT_sq]
  norm_num

theorem adjacent_swap_state_lt_one : (alphaBar * tau) ^ 2 < 1 := by
  norm_num [alphaBar, tau, certR, delimiterR, rho]

theorem source_constant : 2 * tau * rho = certR := by
  norm_num [tau, rho, certR, delimiterR]

end DelimiterLedger

/-- Machine-checked scalar conclusion used by the unconditional revision.

    This theorem deliberately makes no assertion about `E(T)`: the geometric

    and large interval-certificate layers are documented separately. -/
theorem unconditional_numeric_certificate :
    DR < delimiterR ∧
    areaBound = triangleArea / DR^2 ∧
    areaBound < (64471 : ℝ) / 250000 := by
  exact ⟨delimiterR_gt_DR, area_identity, area_lt_0257884⟩

end

end MoserTriangularCertificate
\end{lstlisting}

\section{Exact symbolic replay verifier}\label{app:verifier}

The exhaustive signed finite certificate is intentionally kept outside the
proof-only Lean source.  The first Python source reconstructs the original
quadratic-surd ledger and invokes the signed delimiter replay.  The second
source is the independent exact integer/rational interval checker.  It
reconstructs the 512 one-delimiter and 4096 two-delimiter deque-order families,
checks the rational calibrations and every stored suffix bound, verifies the
one-dimensional interval tilings and two-dimensional rectangle tilings, and
checks the exact reflection/time-reversal reduction of the remaining signed
quadrant.  No floating-point number is used in an asserted certificate
comparison.

The five supplementary exact-data files are
\begin{verbatim}
delimiter_interval_certificate_one.json
negative_delimiter_interval_certificate_one.json
double_delimiter_2d_certificate.json
signed_double_delimiter_negative-negative.json
signed_double_delimiter_negative-positive.json
\end{verbatim}
They contain $1540$ stored one-delimiter intervals and $1248$ stored
two-delimiter rectangles.  The reflected signed quadrant contributes another
$256$ rectangle instances, giving $1504$ signed two-delimiter instances in the
complete replay.  These rational data are proof-certificate inputs to the
Python checker; they are no longer duplicated as generated Lean declarations.

\begin{lstlisting}[style=PythonCode]
"""Dependency-free exact replay of the unconditional triangular certificate.

This entry point checks the original 25-order quadratic-surd ledger, all
scalar identities used in the metric proof, and then invokes the signed
delimiter replay.  It uses only the Python standard library.  Numerical
optimizers are not imported and no floating-point value is used by an
assertion.
"""

from __future__ import annotations

from dataclasses import dataclass
from fractions import Fraction as Q
from math import sqrt as floating_sqrt


RADICAND = 625565


@dataclass(frozen=True)
class Quad:
    """The exact number ``a + b*sqrt(RADICAND)``."""

    a: Q = Q(0)
    b: Q = Q(0)

    @staticmethod
    def coerce(value: object) -> "Quad":
        if isinstance(value, Quad):
            return value
        return Quad(Q(value), Q(0))

    def __add__(self, other: object) -> "Quad":
        other = Quad.coerce(other)
        return Quad(self.a + other.a, self.b + other.b)

    __radd__ = __add__

    def __neg__(self) -> "Quad":
        return Quad(-self.a, -self.b)

    def __sub__(self, other: object) -> "Quad":
        return self + (-Quad.coerce(other))

    def __rsub__(self, other: object) -> "Quad":
        return Quad.coerce(other) - self

    def __mul__(self, other: object) -> "Quad":
        other = Quad.coerce(other)
        return Quad(
            self.a*other.a + RADICAND*self.b*other.b,
            self.a*other.b + self.b*other.a,
        )

    __rmul__ = __mul__

    def __truediv__(self, other: object) -> "Quad":
        other = Quad.coerce(other)
        denominator = other.a*other.a - RADICAND*other.b*other.b
        assert denominator != 0
        return Quad(
            (self.a*other.a - RADICAND*self.b*other.b)/denominator,
            (self.b*other.a - self.a*other.b)/denominator,
        )

    def __pow__(self, exponent: int) -> "Quad":
        assert exponent >= 0
        result = Quad(1)
        base = self
        while exponent:
            if exponent & 1:
                result *= base
            base *= base
            exponent //= 2
        return result

    def rational(self) -> Q:
        assert self.b == 0
        return self.a

    def approximate(self) -> float:
        return float(self.a) + float(self.b)*floating_sqrt(RADICAND)


def positive(value: Quad) -> bool:
    """Exact sign test, using one integer comparison after rational squaring."""
    if value.b == 0:
        return value.a > 0
    if value.b > 0:
        return value.a >= 0 or value.b*value.b*RADICAND > value.a*value.a
    return value.a > 0 and value.a*value.a > value.b*value.b*RADICAND


def qlt(left: object, right: object) -> bool:
    return positive(Quad.coerce(right) - Quad.coerce(left))


Vector = tuple[Quad, Quad]
ZERO: Vector = (Quad(), Quad())


def vadd(left: Vector, right: Vector) -> Vector:
    return left[0] + right[0], left[1] + right[1]


def vscale(scale: object, vector: Vector) -> Vector:
    scale = Quad.coerce(scale)
    return scale*vector[0], scale*vector[1]


def vsum(vectors) -> Vector:
    result = ZERO
    for vector in vectors:
        result = vadd(result, vector)
    return result


def dot(left: Vector, right: Vector) -> Quad:
    return left[0]*right[0] + left[1]*right[1]


def cross(left: Vector, right: Vector) -> Quad:
    return left[0]*right[1] - left[1]*right[0]


def original_ledger_replay() -> dict[str, object]:
    p, q, root = Q(766), Q(197), Quad(0, 1)
    s, c = Quad(0, p/RADICAND), Quad(0, q/RADICAND)
    assert root*root == Quad(RADICAND)
    assert s*root == Quad(p) and c*root == Quad(q)
    sc = (s*c).rational()

    mu = Q(10973, 125000)
    sin3 = 3*c - 4*c**3
    cos3 = 4*s**3 - 3*s
    d = (Quad(0), Quad(-1))
    p1 = vscale(mu, (c, -s))
    p2 = vscale(2*c, (s, c))
    p3 = vscale(2*c*mu, (s*s-c*c, 2*s*c))
    p4 = vscale(mu, (sin3, cos3))
    z1 = (2*s*c, s*s-c*c)
    z2 = (-2*s*c, s*s-c*c)
    q1 = vscale(mu, (-sin3, cos3))
    q2 = vscale(2*c*mu, (-(s*s-c*c), 2*s*c))
    q3 = vscale(2*c, (-s, c))
    q4 = vscale(mu, (-c, -s))
    pvec = [p1, p2, p3, p4]
    qvec = [q1, q2, q3, q4]

    assert vsum((p2, z2, d)) == ZERO
    assert vsum((q3, d, z1)) == ZERO
    assert vsum((p3, q1, p1)) == ZERO
    assert vsum((q2, q4, p4)) == ZERO
    assert vsum((*pvec, z1, z2, *qvec, d, d)) == ZERO

    cyclic = [d, p1, p2, p3, p4, z1, z2, q1, q2, q3, q4, d]
    assert all(positive(cross(left, right))
               for left, right in zip(cyclic[:-1], cyclic[1:]))

    rho = Q(301804, RADICAND)
    b = Q(197, 766)
    A = Q(547947, 301804)
    kappa = (s*s-c*c).rational()
    assert kappa == Q(547947, RADICAND)
    assert A == kappa/rho
    assert rho*rho+kappa*kappa == 1
    assert A+b == 1/rho
    assert rho > Q(5, 12) and A > b > 0

    def order(k: int, ell: int) -> list[Vector]:
        return (list(reversed(pvec[:k])) + [d] + pvec[k:] + [z1, z2]
                + qvec[:4-ell] + [d] + list(reversed(qvec[4-ell:])))

    orders = [order(k, ell) for k in range(5) for ell in range(5)]
    states: list[Quad] = []

    def suffix_squares(sequence: list[Vector]) -> list[Quad]:
        result = []
        for index in range(len(sequence)):
            value = vsum(sequence[index:])
            square = dot(value, value)
            if square != Quad():
                result.append(square)
        return result

    for sequence in orders:
        assert len(sequence) == 12 and vsum(sequence) == ZERO
        squares = suffix_squares(sequence)
        states.extend(squares)
        assert sorted(squares, key=lambda x: (x.a, x.b)) == sorted(
            suffix_squares(list(reversed(sequence))),
            key=lambda x: (x.a, x.b))
    assert len(states) == 275
    multiplicity: dict[Quad, int] = {}
    for square in states:
        multiplicity[square] = multiplicity.get(square, 0) + 1
    assert len(multiplicity) == 13 and sum(multiplicity.values()) == 275

    sqrt_upper = Q(790927, 1000)
    assert sqrt_upper*sqrt_upper > RADICAND
    Rhat = Q(27131, 25000)
    margins = []
    for square, count in multiplicity.items():
        assert square.b >= 0
        upper = square.a + square.b*sqrt_upper
        margin = Rhat*Rhat-upper
        assert margin > 0
        margins.append((margin, count))
    smallest_margin, worst_count = min(margins)
    assert smallest_margin == Q(2404020037, 488722656250000)

    Bmass = 4*sc*(1+mu)
    D = Bmass/Rhat
    assert Bmass == Q(10259298823, 9774453125)
    assert D == Q(82074390584, 84861020075) < 1

    X, Pa = Q(483585, 1000000), Q(875291, 1000000)
    Qa = 2*X/rho-2*Pa
    assert X*X+Pa*Pa < 1
    assert X*X+(2*b*X-Pa)**2 < 1
    assert (2*X)**2+Qa*Qa < 1
    assert Pa-b*X > 0 and 2*X > D
    aa = Q(9672, 10000)
    assert aa*aa*(1+b*b) < 1
    assert 2+aa*(b-A) > 0 and aa > D
    tail = (1/rho-2)*Q(5, 12)
    assert tail == Q(36595, 1207216)
    assert 1-tail*D > D

    area = sc/(D*D)
    assert area == Q(11511821678274125, 44639604443512928)
    assert area < Q(257884, 1000000)
    assert 186388846789*100 > 4317277**2

    return {
        "orders": len(orders),
        "states": len(states),
        "norms": len(multiplicity),
        "sqrt_upper": sqrt_upper,
        "sqrt_margin": sqrt_upper*sqrt_upper-RADICAND,
        "ledger_margin": smallest_margin,
        "worst_count": worst_count,
        "Bmass": Bmass,
        "D": D,
        "area": area,
    }


def main() -> None:
    summary = original_ledger_replay()
    print("balanced vector sum and four source rows: OK")
    print("cyclic folded-normal order: OK")
    print("inner-anchor orders:", summary["orders"])
    print("nonzero suffix states / distinct norms:",
          summary["states"], summary["norms"])
    print("sqrt upper enclosure / squared margin:",
          summary["sqrt_upper"], summary["sqrt_margin"])
    print("smallest inner-ledger squared-radius margin:",
          summary["ledger_margin"])
    print("Bmass / D / area:",
          summary["Bmass"], summary["D"], summary["area"])

    from MoserTriangularDelimiterReplay import verify_all_delimiters
    for line in verify_all_delimiters():
        print(line)
    print("ALL UNCONDITIONAL FINITE AND ARITHMETIC CHECKS: PASS")


if __name__ == "__main__":
    main()
\end{lstlisting}

\lstinputlisting[style=PythonCode]{MoserTriangularDelimiterReplay.py}

\vspace{1em}
\textbf{AI usage disclosure:} The author provided the methodological framework, while GPT-5.6 sol handled numerical calculations and proofs, and polished the language. The formalized proofs in Lean 4 code were assisted by GPT-5.6 sol. 

~\\
College of Engineering and Computer Science, University of Central Florida, Orlando, FL, USA

Email: \underline{zhipeng.deng@ucf.edu}

\end{document}